\documentclass[12pt]{article}

\usepackage{epsfig,amsmath,amsfonts,amssymb,latexsym,color,amsthm,hhline,multirow,caption,subcaption,graphicx,bm, bbm, cancel, soul}
\usepackage{enumerate}

\usepackage{tikz}

\newtheorem{theorem}{Theorem}

\newtheorem{prop}{Proposition}
\newtheorem{lemma}{Lemma}
\newtheorem{remark}{Remark}
\newtheorem{corr}{Corollary}

\newcommand{\E}{{\mathbb E}}
\newcommand{\RR}{{\mathbb R}}
\newcommand{\Rd}{\mathbb{R}^d}

\newcommand{\Z}{{\mathbb Z}}
\newcommand {\PP}{{\mathbb P}}
\newcommand{\sss}{\scriptscriptstyle}
\newcommand{\Zd}{\mathbb{Z}^d}

\newcommand{\cG}{\mathcal{G}}
\newcommand{\cU}{\mathcal{U}}
\newcommand{\cV}{\mathcal{V}}
\newcommand{\cGbi}{\cG_{\sss{\rm{bi}}}}
\newcommand{\sm}{s_{\sss{\rm{max}}}}

\begin{document}

\captionsetup[subfigure]{justification=centering}

\title{Geometric random intersection graphs with general connection probabilities}\parskip=5pt plus1pt minus1pt \parindent=0pt
\author{Maria Deijfen\thanks{Department of Mathematics, Stockholm University, {\tt mia@math.su.se}.}\and Riccardo Michielan\thanks{Department of Electrical Engineering, Mathematics and Computer Science, University of Twente, {\tt r.michielan@utwente.nl}.}}
\date{June 2023}
\maketitle

\begin{abstract}
\noindent Let $\cV$ and $\cU$ be the point sets of two independent homogeneous Poisson processes on $\Rd$. A graph $\cG_\cV$ with vertex set $\cV$ is constructed by first connecting pairs of points $(v,u)$ with $v\in\cV$ and $u\in\cU$ independently with probability $g(v-u)$, where $g$ is a non-increasing radial function, and then connecting two points $v_1,v_2\in\cV$ if and only if they have a joint neighbor $u\in\cU$. This gives rise to a random intersection graph on $\Rd$. Local properties of the graph, including the degree distribution, are investigated and quantified in terms of the intensities of the underlying Poisson processes and the function $g$. Furthermore, the percolation properties of the graph are characterized and shown to differ depending on whether $g$ has bounded or unbounded support.
\vspace{0.3cm}


\vspace{0.2cm}

\end{abstract}

\section{Introduction}\label{sec:intro}

Random intersection graphs have been popular in network modeling to describe networks arising from bipartite structures. In general, an intersection graph is constructed by assigning each vertex a subset of some auxiliary space, and then connecting two vertices if their subsets intersect; see \cite{Singer}. In the context of network modeling, the auxiliary space typically consists of an additional vertex set, so that the graph is constructed based on two disjoint vertex sets. A bipartite graph is then generated by connecting vertices and auxiliary vertices in some random way and, in the second step, connecting two vertices if there exist an auxiliary vertex to which they are both connected; see e.g.\ \cite{B13,GJ01,KSS99,N03,S04}. In applications, the vertex set can for instance consist of individuals while the auxiliary vertices represent social groups, so that two individuals are connected when they share a social group. Other examples include communication units connected via cell towers, and scientists related through joint papers. However, we will throughout refer to the auxiliary vertices as \emph {groups}.

Most models of random intersection graphs in the literature are non-geometrical. From an empirical perspective however there are many examples where geographical aspects are likely to play a role for the formation of a network of this type. Children are for instance more likely to join sports clubs close to their home, and wireless communication tend to take place through cell towers in the geographical vicinity. Moreover, a notion of proximity can be induced by other features, such as individuals sharing common interests or being part of the same age group, when these features are modeled on a proper space. We hence study a model where both vertices and groups have locations in space. Specifically, we model the two vertex sets as realizations of independent homogeneous Poisson processes on $\RR^d$. A version of such a model has been studied under the name of AB Poisson Boolean model in the special case when a vertex connects to a group if and only if their distance is less than some constant; see \cite{IY12}. Our purpose is to extend this model to more general connection probabilities, corresponding to the generalization of the renowned Poisson Boolean model to the random connection model; see \cite{MR}.

The number of groups shared by two given vertices in our model turns out to be Poisson distributed with a parameter that depends on the distance between the vertices. In non-spatial versions of the random intersection graph, the probability for two given vertices to share more than one group is typically very small, so that the fraction of vertices that share two or more groups is negligible. The behavior of our model is in many cases more realistic, since multiple joint groups are not unlikely for vertices that are geographically close. The downside is that geometry induces heavy dependencies in the edge formation and makes it difficult to characterize the degree distribution explicitly. Indeed, the degree of a vertex is not simply the sum of the sizes of all groups the vertex is a member of, but overlaps between the groups need to be subtracted. However, we provide a characterization of the degree distribution and illustrate its behavior with simulations. We also study the percolation properties of the graph and, when the connection probabilities have unbounded support, we observe a qualitatively different behavior compared to the AB Poisson Boolean model.

\subsection{Description of the model}

Let $\cV$ and $\cU$ denote the point sets of two independent homogeneous Poisson processes on $\RR^d$ with intensities $\lambda$ and $\mu$, respectively. The sets $\cV$ and $\cU$ represent the positions of vertices and groups, respectively. We will often assume that $\cV$ has a point at the origin. This is known as the Palm version of the process and it is well-known that, for a Poisson process, this has the same distribution as the original process with an added point at the origin; see \cite{DVJ}. We think of the origin point as a typical point of the process.

Let $|\cdot|$ denote the Euclidean norm. The connection probabilities will be based on a function $g(x):\RR^d\to[0,1]$, with the properties that $g(x)=g(y)$ if $|x|=|y|$ and $g(x)\leq g(y)$ if $|x|\geq |y|$, that is, $g$ is a non-increasing radial function. Our graph is constructed in two steps as follows:

\begin{itemize}
\item[(i)] Each pair of points $(v,u)$ with $v\in\cV$ and $u\in\cU$ are connected independently with probability $g(v-u)$. This gives rise to a bipartite graph $\mathcal{G}_{\sss \rm{bi}}$ with vertex set $\cV\cup \cU$.
\item[(ii)] Two points $v_1,v_2\in \cV$ are then connected if and only if there exist $u\in\cU$ such that both $v_1$ and $v_2$ are connected to $u$ in $\mathcal{G}_{\sss\rm{bi}}$. The resulting random intersection graph is denoted by $\mathcal{G}_{\cV}$ and has vertex set $\cV$.
\end{itemize}

It is straightforward to see that the $\cU$-points that the origin vertex in $\cV$ is connected to in $\mathcal{G}_{\sss \rm{bi}}$ constitute an inhomogeneous Poisson process with intensity $\mu g(x)$; c.f.\ \cite[Proposition 1.3]{MR}. The expected number of groups that the origin is a member of is hence $\mu\int_{\RR^d}g(x)dx$ and, to make the model non-trivial, we will throughout assume that
\begin{equation}\label{gass}
0<\int_{\RR^d}g(x)dx<\infty.
\end{equation}
The above integral is the $L_1$-norm of $g$ and we write $\int_{\RR^d}g(x)dx=||g||$.

Note that the edges in $\mathcal{G}_{\cV}$ are not independent, since the presence of an edge between two vertices gives information about the presence of a group at a suitable location causing the connection and this information in turn affects the presence of other edges. This causes a substantial dependence between the connections of two vertices located close to each other, while the dependence becomes weaker for vertices far apart. The fact that the edges are not independently present is an important difference compared to the standard random connection model, where each pair of points $x$ and $y$ in a homogeneous Poisson process is connected independently with probability $g(x-y)$; see \cite{MR}.

A special case of the random connection model is when $g(x)=\mathbbm{1}_{|x|<r}$ for some $r\in\RR$. This is the well-known Poisson Boolean model, where a ball with radius $r/2$ is placed at each Poisson point and two points are then connected if their balls intersect, that is, if they are within distance $r$. The percolation properties of the model in the present paper for this particular choice of $g$ was studied in \cite{IY12}. The model was then motivated as a continuum analog of so called AB percolation on lattices, where sites are independently assigned the mark A or B and only edges between sites with different marks are kept; see e.g. \cite{H80,WA87,WP03} and references therein.

\subsection{Basic properties}

First consider the number of groups that the origin vertex share with another vertex $v\in\cV$. As pointed out above, the groups that the origin is connected to form an inhomogeneous Poisson process with intensity $\mu g(x)$. This process can be thinned further to include only those groups to which also $v$ belongs. The probability that $v$ belongs to a group located at $x$ is $g(v-x)$, and the doubly thinned process containing the groups where both the origin and $v$ are members is hence an inhomogeneous Poisson process with intensity $\mu g(x)g(v-x)$.

\begin{prop}\label{p:joint_groups}
The groups shared by vertices located at 0 and $v$ form an inhomogeneous Poisson process on $\RR^d$ with intensity $\mu g(x)g(v-x)$.
\end{prop}

For completeness, we include a formal proof of this in Section 2. Let $N_{0,v}$ denote the number of groups shared by 0 and $v$. Then $N_{0,v}$ is Poisson distributed with parameter $\mu\int_{\RR^d}g(x)g(v-x)dx$. We write
\begin{equation}\label{f}
f(v)=\int_{\RR^d}g(x)g(v-x)dx,
\end{equation}
that is, $f$ is the convolution of $g$ with itself, and note that $f(v)\leq \int_{\RR^d}g(x)dx<\infty$ for all $v$. Also note that $f$ inherits the properties of being a non-increasing radial function from $g$. Two vertices are connected in the random intersection graph if they share at least one group, which gives an expression for the connection probability by computing $\PP(N_{0,v}>0)$.

\begin{corr}\label{cor:conn_prob}
The probability that the vertices $0$ and $v$ in $\mathcal{G}_{\cV}$ are connected is given by $p_{0,v}:=1-e^{-\mu f(v)}$.
\end{corr}

Note that $p_{0,v}=0$ if $f(v)=0$. This occurs if $g$ has bounded support and $v$ does not belong to the support of $g$. If $f(v)>0$, on the other hand, it follows from the expression for $p_{0,v}$ that the connection probability can be made as close to 1 as we wish by increasing $\mu$. Intuitively, if it is possible for $0$ and $v$ to connect, the probability that they actually do so can be made large by increasing the group intensity.

\begin{corr}\label{cor:to1}
If $f(v)>0$, then $p_{0,v}\to 1$ as $\mu\to\infty$.
\end{corr}

We now turn to the degree distribution in the random intersection graph. Let $D$ denote the degree of the origin vertex in $\mathcal{G}_{\cV}$, that is,
\begin{equation}\label{Dsum}
D=\sum_{v\in\cV}\mathbbm{1}_{\{0\leftrightarrow v\}},
\end{equation}
where $0\leftrightarrow v$ means that 0 and $v$ are connected in $\mathcal{G}_{\cV}$. Since the edges in $\mathcal{G}_{\cV}$ are not independent, it is difficult to characterize the degree distribution precisely, but an upper bound is easily obtained by ignoring overlaps between the groups that the origin is a member of. Below, the variable $N$ represent the number of groups of the origin and $\{X_i\}_{i\geq 1}$ represent their sizes.

\begin{prop}\label{p:degree_bound}
The degree $D$ is stochastically dominated by $\sum_{i=1}^NX_i$, where $\{X_i\}_{i\geq 1}$ are i.i.d.\ Poisson variables with mean $\lambda||g||$ and $N$ is a Poisson variable with mean $\mu||g||$, independent of $\{X_i\}_{i\geq 1}$.
\end{prop}

We give the short proof of Proposition \ref{p:degree_bound} in Section 2. In Section 3, we give some examples of choices of connection functions $g$, and illustrate the corresponding degree distributions by aid of simulations. Note that it follows from Proposition \ref{p:degree_bound} that the degree distribution has exponentially decaying tail. In order to obtain a power law tail, the connection probability has be made inhomogeneous; see Section \ref{sec:further} for further comments on this. Even though the full degree distribution is complicated to characterize, the expected degree can be computed.

\begin{prop}\label{p:edegree}
We have that $\E[D]=\lambda\int_{\RR^d}\left(1-e^{-\mu f(y)}\right)dy$.
\end{prop}

\begin{proof}
An edge indicator in the expression \eqref{Dsum} for the degree of the origin is a Bernoulli random variable with success probability $p_{0,v}$. Hence $\E[D]=\sum _{v\in\cV}p_{0,v}$. This sum is the expected number of points in a thinned version of a Poisson process with rate $\lambda$ where a point at $v$ is kept with probability $p_{0,v}$, that is, in an inhomogeneous Poisson process with intensity $\lambda p_{0,v}$. It follows that $\E[D]=\lambda\int_{\RR^d}p_{0,v}dv$, as claimed.
\end{proof}

Since $1+r<e^r$ for $r\in\RR$, we can upper bound
$$
\E[D]\leq \lambda\mu\int_{\RR^d}f(y)dy = \lambda\mu||g||^2.
$$
The expected degree can hence be made small by decreasing either $\lambda$ or $\mu$. It also follows from the expression for $\E[D]$ in Proposition \ref{p:edegree} that the expected degree grows large linearly with the intensity $\lambda$ of the vertex set. Indeed, groups are likely to be large if $\lambda$ is large, and the origin has a positive probability of belonging to at least one group. When $\lambda$ is fixed and $\mu$ increases, on the other hand, the expected degree grows to infinity if and only if $g$ has unbounded support. Indeed, if vertices beyond a certain range are unreachable, the degree stays bounded, but if all vertices are potentially within reach, the expected degree becomes large when the number of groups increases, since the connection probability between any two vertices then comes close to 1; c.f.\ Corollary \ref{cor:to1}. To see this from the expression for $\E[D]$, let $\ell(r)$ denote the volume of a $d$-dimensional ball with radius $r$, and define $r_s=\sup\{|v|:f(v)>s\}$. The integral in the expression for $\E[D]$ can then be bounded as
$$
\ell(r_{1/\mu})(1-e^{-1})\leq \int_{\RR^d}\left(1-e^{-\mu f(v)}\right)dv\leq \ell(r_0),
$$
where $r_0<\infty$ when $g$ has bounded support, while $r_{1/\mu}\to\infty$ as $\mu\to\infty$ when $g$ has unbounded support. The different behavior of the expected degree for large $\mu$ depending on the support of $g$ gives rise to a different phase transition for percolation, as we will see in the next subsection.

\subsection{Percolation phase transition}

We now turn to the question of whether there exists an infinite component in the random intersection graph $\mathcal{G}_{\cV}$. To this end, let $C$ denote the number of vertices in the component of the origin vertex. A straightforward coupling argument shows that the percolation function $\theta(\lambda,\mu,g)=\PP(C=\infty)$ is increasing in $\mu$, and we define
\begin{equation}\label{muc}
\mu_c=\mu_c(\lambda,g)=\sup\{\mu:\theta(\lambda,\mu,g)=0\}.
\end{equation}
A graph is said to percolate if it contains an infinite component, and it follows from ergodicity that $\cG_{\cV}$ percolates with probability 0 or 1 whenever $\mu<\mu_c$ or $\mu>\mu_c$.

We will fix $\lambda$ and investigate $\mu_c(\lambda,g)$. Another option would be to fix $\mu$ and investigate $\lambda_c(\mu,g)$ defined in an analogous way. This however would give rise to a qualitatively similar picture. Specifically, we have that $\mu_c(a,g)=\lambda_c(a,g)$ for all $a >0$. To see this, let $\cG_{\cU}$ be the graph obtained from $\cG_{\sss\rm{bi}}$ in a similar way as $\cG_{\cV}$, but projecting on the point set $\cU$ instead of $\cV$. That is, the sets $\cV$ and $\cU$ switch roles in part (ii) of the construction of the graph, so that two groups are connected if and only if there is a vertex that is a member of both of them. We have the following equivalence result.

\begin{lemma}\label{le:equivalence}
If $||g||<\infty$, then $\{\cG_{\cV}$ percolates $\}$ $\Leftrightarrow$ $\{\cG_{\sss\rm{bi}}$ percolates $\}$ $\Leftrightarrow$ $\{\cG_{\cU}$ percolates $\}$.
\end{lemma}

\begin{proof}
Just note that, when $||g||<\infty$, each vertex (group) is connected to an almost surely finite number of groups (vertices) in $\cGbi$. Hence percolation in $\cGbi$ is equivalent to the existence of an infinite path where groups and vertices alternate. The existence of such a path implies percolation in both $\cG_{\cV}$ and $\cG_{\cU}$. On the other hand, an infinite component in $\cG_{\cV}$ (or $\cG_{\cU}$) implies percolation in $\cGbi$, since all vertices in a given component of $\cG_{\cV}$ (or $\cG_{\cU}$) belongs to the same component in $\cGbi$.
\end{proof}

Note that $\cG_{\cU}$ is equivalent to a graph where the roles of $\lambda$ and $\mu$ are interchanged, so that $\mu$ is the intensity of the vertex set and $\lambda$ of the auxiliary vertex set. Hence it follows from the above lemma that, if there is percolation in $\cG_\cV$ for $\lambda=a$ and $\mu=b$, then there is also percolation for $\lambda=b$ and $\mu=a$. It is thus enough to fix $\lambda$ and vary $\mu$.

The percolation properties of the standard random connection model (with one single vertex set with intensity $\lambda$) is described in \cite{MR}. It is shown that, for $d\geq 2$, there is a non-trivial critical value $\tilde{\lambda}_c(g)\in(0,\infty)$ such that the graph percolates if $\lambda>\tilde{\lambda}_c(g)$ while it does not percolate if $\lambda<\tilde{\lambda}_c(g)$. Write $\tilde{\lambda}_c(r)$ for the critical value in the standard Poisson Boolean model, that is, when $g(x)=\mathbbm{1}_{|x|<r}$ for some $r>0$. In \cite{IY12}, the same model as in the present paper is analyzed in the Poisson Boolean setting. Let $\mu_c(\lambda,r)$ denote the critical value for this model, as defined in \eqref{muc}. It is shown that $\mu_c(\lambda,r)=\infty$ when $\lambda<\tilde{\lambda}_c(2r)$, while $\mu_c(\lambda,r)\in(0,\infty)$ when $\lambda$ is sufficiently large, where $\lambda>\tilde{\lambda}_c(2r)$ suffices for $d=2$.

Our main result states that, when $g$ has unbounded support, there is a non-trivial phase transition in $\mu$ in the random intersection graph $\cG_\cV$ for any fixed value of $\lambda$. When $g$ has bounded support, the qualitative picture is the same as for the Poisson Boolean version. We denote $\sm=\sup\{|x|:g(x)>0\}$.

\begin{theorem}\label{th:main}
Consider the graph $\cG_\cV$ with $||g||<\infty$ in dimension $d\geq 2$.
\begin{itemize}
\item[(a)] If $g$ has unbounded support, then $\mu_c(\lambda,g)\in(0,\infty)$ for any $\lambda>0$.
\item[(b)] If $g$ has bounded support, then $\mu_c(\lambda,g)=\infty$ for $\lambda<\tilde{\lambda}_c(2\sm)$, while $\mu_c(\lambda,g)\in(0,\infty)$ if $\lambda$ is sufficiently large, where $\lambda>\tilde{\lambda}_c(2\sm)$ suffices in $d=2$.
\end{itemize}
\end{theorem}

Non-percolation for small values of $\mu$ is established by dominating the exploration of our graph with a subcritical branching process. When $g$ has bounded support, our model is dominated by a Poisson Boolean model with $r=\sm$, so in this case non-percolation is also an immediate consequence of the fact that the Poisson Boolean version does not percolate for small values of $\mu$. Percolation for large values of $\mu$ when $g$ has unbounded support is established by introducing a 1-dependent percolation process, which is shown to percolate for large values of $\mu$ due to Corollary \ref{cor:to1}. The argument can be adapted to show percolation also when $g$ has bounded support, provided $\lambda$ is sufficiently large. Finally, percolation for large values of $\mu$ as soon as $\lambda>\tilde{\lambda}_c(2\sm)$ is obtained by a generalization of the argument in \cite{IY12}.

Below we give some suggestions for further work. The rest of the paper is then organized so that the proofs are collected in Section 2, and some examples and numerical illustrations are given in Section 3.

\subsection{Further work}\label{sec:further}

In this section we list some open problem and possible generalizations of the model.

\textbf{Critical $\lambda$ for $d\geq 3$.} The percolation picture described in Theorem \ref{th:main} leaves open if percolation is possible for any $\lambda>\tilde{\lambda}_c(2\sm)$ in dimension $d\geq 3$ when $g$ has bounded support. The proof in $d=2$ does not extend to higher dimensions and it would be interesting to understand if the statement is still true or if there is some interval $(\tilde{\lambda}_c(2\sm),\bar{\lambda})$ where percolation is not possible.

\textbf{Percolation in one dimension.} Theorem \ref{th:main} covers only $d\geq 2$. When $g$ has bounded support, it is not hard to see that percolation is not possible in $d=1$. The case when $g$ has unbounded support is more subtle. Results for the standard random connection model indicate that the decay of $g$ is important. In particular, if $g(x-y)\sim |x-y|^{-\delta}$, then there is a non-trivial phase transition in the standard random connection model for $\delta\in(1,2)$, while percolation is not possible for $\delta>2$ nor if $g$ decays faster than polynomially. This is shown for long-range percolation on $\Z$ in \cite{NS86,S83} and follows in the Poisson setting from the more general results in \cite{GLM21,GLM23}. It could be worthwhile to analyze the present model in more detail in $d=1$ for $g$ with unbounded support and derive conditions for percolation and non-percolation, respectively.

\textbf{Clustering.} One reason why random intersection graphs have been popular in the non-spatial setting is that they give rise to clustering in the graph, manifested in the presence of a large number of triangles; see e.g.\ \cite{B13,DK09,N03}. Empirical networks arising from social interaction often exhibit clustering, since two people with a joint friend often make friends also with each other. In non-spatial random intersection graphs, clustering arises as a consequence of the fact that, if two vertices both share a group with a third vertex, then there is a non-trivial probability that they all share the same group, which in that case means that they form a triangle. Spatial graphs typically exhibit clustering also without the intersection effect, since connecting vertices based on their distance in itself induces clustering. Therefore, clustering is not a motivating factor for the spatial random intersection graph as much as in the non-spatial case. However, it could still be interesting to analyze the model in the present paper in this respect and compare for instance to the standard random connection model.

\textbf{Inhomogeneous versions.} Empirical networks often exhibit large elements of inhomogeneity manifested for instance in power-law degree distributions. In order to achieve this in the present model, the edge probabilities would have to be made more variable, for instance by assigning random weights to the vertices and let the edge probabilities be determined by a combination of the weights and the distance between the vertices. Spatial graph models of this type without the intersection effect have been extensively studied the last few years; see e.g.\ \cite{BKL19,DHG13,GHMM22,GLM21,H17,Y06}. Inhomogeneous versions of non-spatial random intersection graphs have been studied in \cite{B17,BD13,DK09}. The present model could also be extended to such a setting, and one could then analyze the effect of the weights on the degree distribution and percolation properties.

\section{Proofs}

We first confirm that the number of groups shared by two vertices at 0 and $v$ constitute an inhomogeneous Poisson process, as claimed in Proposition \ref{p:joint_groups}.

\begin{proof}[Proof of Proposition \ref{p:joint_groups}]
Fix a bounded Borel set $B \subset \mathbb{R}^d$, and let $N_{0,v}(B)$ denote the number of groups in $B$ shared by $0$ and $v$. Clearly $N_{0,v}(B)$ and $N_{0,v}(B')$ are independent for disjoint sets $B$ and $B'$. We have to prove that $N_{0,v}(B)$ is Poisson distributed with parameter $\mu \int_B g(x)g(v-x)dx$. Write $U(B)$ for the number of groups in $B$ and note that $U(B)$ is Poisson distributed with parameter $\mu\cdot\ell(B)$, where $\ell(\cdot)$ denotes Lebesgue measure on $\RR^d$. The groups are uniformly distributed over $B$ and the probability that both 0 and $v$ are members of a given group is hence given by
$$
p_{0,v}(B) = \frac{\int_B g(x)g(v-x) dx}{\ell(B)}.
$$
We then have that
$$
N_{0,v}(B) \stackrel{d}{=} \sum_{i = 1}^{U(B)} Z_i,
$$
where $\{Z_i\}$ are i.i.d.\ Bernoulli variables with parameter $p_{0,v}(B)$, independent of $U(B)$. The generating function of $N_{0,v}(B)$ is thus given by
$$
\E\left[s^{N_{0,v}(B)}\right] =  \E\left[(1+p_{0,v}(B)(s-1))^{U(B)}\right] = e^{(s-1)\mu\int_B g(x)g(v-x) dx},
$$
which we recognize as the generating function of the stated Poisson distribution.
\end{proof}

Next, we observe that the degree of a given vertex in $\cG_\cV$ can be dominated as described in Proposition \ref{p:degree_bound}.

\begin{proof}[Proof of Proposition \ref{p:degree_bound}]
Let $N$ denote the number of groups that the origin is connected to. Then $N$ is Poisson distributed with parameter $\mu||g||$. Denote the groups of the origin by $u_1,\ldots,u_N$. The degree of the origin is given by the total number of vertices connected to these groups. We will dominate this set by aid of a superposition of independent inhomogeneous Poisson processes. First consider $u_1$ and let $\mathcal{P}_1$ be an inhomogeneous Poisson process with intensity function $\lambda g(x-u_1)$, representing the vertices connected to $u_1$ (apart from the origin). Then consider $u_2$ and let $\mathcal{P}_2$ be an inhomogeneous Poisson process, independent of $\mathcal{P}_1$ and with intensity function $\lambda (1-g(x-u_1))g(x-u_2)$, representing the vertices that are connected to $u_2$, but not to $u_1$. Similarly, for $i\geq 3$, the process $\mathcal{P}_i$ contains the points connected to $u_i$ but not to $u_1,\ldots,u_{i-1}$. For any $i\geq 1$, the intensity function of $\mathcal{P}_i$ is dominated by $\lambda g(x-u_i)$, and hence each process $\mathcal{P}_i$ can be coupled to an inhomogeneous Poisson process $\bar{\mathcal{P}}_i$ with rate $\lambda g(x-u_i)$ in such a way that the point set of the former is a subset of the latter. Write $X_i$ for the total number of points in $\bar{\mathcal{P}}_i$. Then $X_i$ is Poisson distributed with parameter $\lambda\int_{\RR^d}g(x-u_i)dx=\lambda||g||$, and the proposition is established.
\end{proof}

We now turn to the percolation properties of our graph, and split the proof of Theorem \ref{th:main} into a few lemmas, the first one stating that the graph does not percolate for small values of $\mu$.

\begin{lemma}\label{le:no_perc}
If $||g||<\infty$ and $d\geq 1$, then $\mu_c(\lambda,g)>0$ for any $\lambda>0$.
\end{lemma}

\begin{proof}
Fix $\lambda>0$. We will explore the graph in generations, starting from the origin, and show that this exploration process is almost surely finite when $\mu$ is small. Specifically, let $G_n$ denote the number of vertices at graph distance $n$ from the origin. We claim that the process $\{G_n\}_{n\geq 1}$ is dominated by a branching process with offspring distribution $\sum_{i=1}^NX_i$, where $N$ is a Poisson variable with parameter $\mu||g||$ and $\{X_i\}$ are i.i.d.\ copies of a Poisson random variable $X$ with parameter $\lambda||g||$, independent of $N$. Indeed, it follows from Proposition \ref{p:degree_bound} that $G_1$ is bounded in the claimed way, since $G_1$ is equal to the degree of the origin.

For $n\geq 2$, we iterate the same construction as in the proof of Proposition \ref{p:degree_bound} and extend it to the groups. To be more precise, let $\{v^{\sss(n-1)}_j:j=1,\ldots, G_{n-1}\}$ denote the vertices at graph distance $n-1$ from the origin and write $N_j^{\sss (n-1)}$ for the number of groups that $v^{\sss(n-1)}_j$ is a member of but where no vertex of any previous generation is a member and, for $j\geq 2$, also none of the vertices $v^{\sss(n-1)}_1,\ldots,v^{\sss(n-1)}_{j-1}$ in the same generation is a member. Those groups constitute an inhomogeneous Poisson process, independent of succeeding quantities, with an intensity function that is dominated by $\mu g(x-v^{\sss(n-1)}_j)$, where the bound is obtained by ignoring factors stemming from the exclusion of groups that have already been visited by the process. Hence the process can be coupled to an inhomogeneous Poisson process where the total number of points $\bar{N}_j^{\sss (n-1)}$ is Poisson distributed with parameter $\mu||g||$ in such a way that $N_j^{\sss (n-1)}\leq \bar{N}_j^{\sss (n-1)}$. Note that $\bar{N}_j^{\sss (n-1)}\stackrel{d}{=}N$.

Denote the groups that vertex $v^{\sss(n-1)}_j$ is a member of by $\{u^{\sss(n-1)}_{j,i}:i=1,\ldots, N_j^{\sss (n-1)}\}$ and let $X^{\sss(n-1)}_{j,i}$ be the number of vertices that are connected to the group $u^{\sss(n-1)}_{j,i}$ and that have not been visited before in the process. Those vertices constitute an inhomogeneous Poisson process, independent of previous quantities, with an intensity that is bounded by $\lambda g(x-u^{\sss(n-1)}_{j,i})$. We can hence couple the process to an inhomogeneous Poisson process where the total number of points $\bar{X}^{\sss(n-1)}_{j,i}$ is Poisson distributed with parameter $\mu||g||$ in such a way that $X^{\sss(n-1)}_{j,i}\leq \bar{X}^{\sss(n-1)}_{j,i}$. Note that $\bar{X}^{\sss(n-1)}_{j,i}\stackrel{d}{=}X$.

The contribution of each vertex $v^{\sss(n-1)}_j$ to the next generation is $\sum_{i=1}^{N_j^{\sss (n-1)}}X^{\sss(n-1)}_{j,i}$. Above, we have shown that this can be upper bounded by coupling the variables to independent ones with the same distributions as $N$ and $X$, respectively. It follows that the exploration of the graph can be coupled to a branching process with offspring distribution $\sum_{i=1}^NX_i$ in such a way that the total number of vertices in the component of the origin is bounded by the total progeny in the branching process. The offspring mean is $\lambda\mu ||g||^2$ and, if $\lambda$ is fixed and $||g||<\infty$, this can be made smaller than 1 by picking $\mu$ small. The total progeny of the branching process is then finite with probability 1, and we conclude that the component of the origin is almost surely finite.
\end{proof}

We continue by establishing percolation for large values of $\mu$ when $g$ has unbounded support.

\begin{lemma}\label{le:perc_ubdd}
If $g$ has unbounded support and $||g||<\infty$, then $\mu_c(\lambda,g)<\infty$ for any $\lambda>0$ and $d\geq 2$.
\end{lemma}

The proof uses a result on domination of $k$-dependent bond percolation by product measure, which is an immediate consequence of the general result in \cite{LSS}. Let $E$ denote the set of nearest neighbor edges of the $\Zd$-lattice. A process $\{Y_e\}_{e\in E}$ is $k$-dependent if, for any two sets $E_1,E_2\subset E$ at $l_\infty$-distance at least $k$, the variables $\{Y_e\}_{e\in E_1}$ and $\{Y_e\}_{e\in E_2}$ are independent.

\begin{lemma}[Liggett, Schonmann, Stacey (1997)]\label{le:LSS}
For any $d\geq 2$ and $k\geq 1$, there exists $p_c=p_c(d,k)<1$ such that, for any $k$-dependent process $\{Y_e\}_{e\in E}$ with $\PP(Y_e=1)=1-\PP(Y_e=0)>p_c$, the $1'$s in $\{Y_e\}_{e\in E}$ percolate almost surely.
\end{lemma}

\begin{proof}[Proof of Lemma \ref{le:perc_ubdd}]
We will define a 1-dependent bond percolation model on $\Zd$ with the property that percolation in this model implies the existence of an infinite component in our graph $\cG_\cV$. The marginal probabilities in the percolation model can be made arbitrarily close to 1 by picking $\mu$ large, which will imply percolation. To define the model, for $z\in\Zd$, let $C_z^m$ denote the cube with side length $m/2$ centered at $mz$ and, for neighboring sites $z$ and $z'$, let $C_{z,z'}^m$ denote the cube with side length $m/2$ centered at $(z+z')/2$, in between $C_z^m$ and $C_{z'}^m$; see Figure \ref{fig:lattice_supercritical}. Note that the maximal distance from a point in $C_z^m$ to a point in $C_{z,z'}^m$ is bounded from above by, for instance, $dm$.

\begin{figure}
    \centering
    \includegraphics[width=0.50\textwidth]{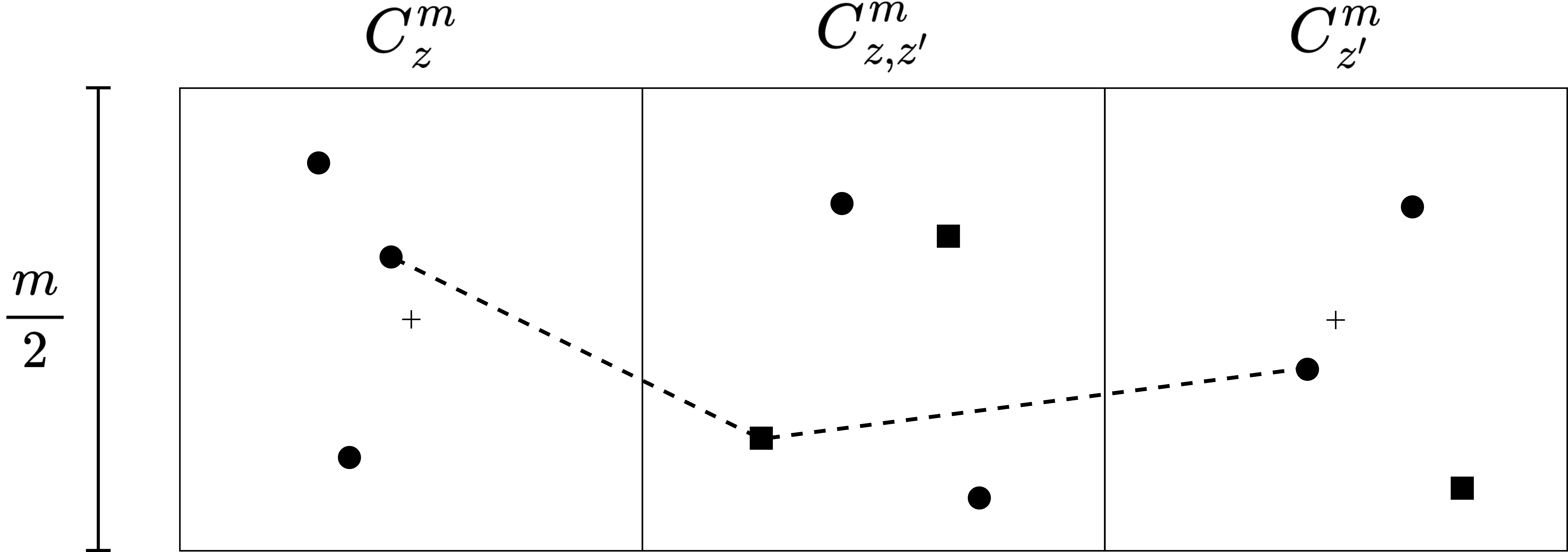}
    \caption{Geometrical construction in the proof of Lemma \ref{le:perc_ubdd}.}
    \label{fig:lattice_supercritical}
\end{figure}

Say that a site $z$ is open if $C_z^m$ contains at least one point of the vertex set $\cV$. We then want to declare an edge between two sites $z$ and $z'$ open if both sites are open and, in addition, there is a connection between a vertex in $C_z^m$ and a vertex in $C_{z'}^m$. This would however lead to long-range dependencies between edges, since we would not have control over the location of the group(s) that causes the vertices to be connected. To circumvent this, we add a restriction on the location of the connecting group. For an open cube $C_z^m$, let $v_z$ denote the vertex closest to the center $zm$ of the cube. An edge between two neighboring sites $z$ and $z'$ is said to be open if both sites are open and there exists a group $u$ in the intermediate cube $C_{z,z'}^m$ such that $v_z$ and $v_{z'}$ are both connected to $u$. Note that percolation of the open edges implies the existence of an infinite component in $\cG_\cV$.

The states of adjacent edges are not independent, since the edges share a vertex. However, the states of edges that do not share a vertex are defined based on Poisson configurations in disjoint regions and therefore independent. The classification of the edges is hence 1-dependent, and it follows from Lemma \ref{le:LSS} that the open edges percolate if the probability that an edge is open is larger than some value $p_c<1$. We have that
\begin{equation}\label{site_open}
\PP(z \mbox{ open})=\PP(C_z^m\cap \cV\neq\emptyset)=1-e^{-\lambda \ell(C_z^m)}.
\end{equation}
This probability can be made larger than $p_c^{1/3}$ by picking $m$ large. Fix such an $m$. For two open sites $z$ and $z'$, by Proposition \ref{p:joint_groups}, the groups where both $v_z$ and $v_{z'}$ are members constitute an inhomogeneous Poisson process with intensity function $\mu g(x-v_z)g(x-v_{z'})$. The number of such groups in the intermediate cube $C_{z,z'}^m$ is hence Poisson distributed with parameter
$$
\mu\int_{C_{z,z'}^m}g(x-v_z)g(x-v_{z'})dx\geq \mu \ell(C_{z,z'}^m)g(dm)^2,
$$
where the bound follows since $|x-v_z|\leq dm$ and $|x-v_{z'}|\leq dm$ for $x\in C_{z,z'}^m$. Hence
\begin{equation}\label{edge_open_cond}
\PP(\mbox{edge $(z,z')$ open}\,|\mbox{$z$ and $z'$ open})\geq 1-e^{-\mu \ell(C_{z,z'}^m)g(dm)^2}.
\end{equation}
If $g$ has unbounded support, then $g(dm)>0$ and thus the above probability can be made larger than $p_c^{1/3}$ by picking $\mu$ large. For $m$ and $\mu$ chosen in this way we have that
\begin{equation}\label{edge_open}
\PP(\mbox{edge $(z,z')$ open})=\PP(\mbox{$z$ open})\cdot\PP(\mbox{$z'$ open})\cdot\PP(\mbox{$(z,z')$ open}\,|\mbox{$z$ and $z'$ open})> p_c.
\end{equation}
If follows that the open edges percolate, which enforces an infinite component in $\cG_\cV$.
\end{proof}

Next we observe that an analogue of Lemma \ref{le:perc_ubdd} holds also when $g$ has bounded support, but with the additional requirement that $\lambda$ must be sufficiently large.

\begin{lemma}\label{le:perc_bdd}
In $d\geq 2$, if $g$ has bounded support and $||g||<\infty$, then  $\mu_c(\lambda,g)<\infty$ when $\lambda$ is sufficiently large.
\end{lemma}

\begin{proof}
We adapt the proof of Lemma \ref{le:perc_ubdd} to the case when $g$ has bounded support. Recall that $\sm=\sup\{|x|:g(x)>0\}$. We use the same construction as in the proof Lemma \ref{le:perc_ubdd} but with $m=\sm/(2d)$. The probability \eqref{site_open} that a cube $C_z^m$ contains at least one vertex from $\cV$ can then be made larger than $p_c^{1/3}$ by picking $\lambda$ large, while the conditional edge probability \eqref{edge_open_cond} can be made larger than $p_c^{1/3}$ by picking $\mu$ large, since $g(\sm/2)>0$. It follows as in \eqref{edge_open} that the probability of an edge being open is larger than $p_c$, and we conclude that our graph percolates almost surely.
\end{proof}

Finally, we adapt the proof of \cite[Theorem 2.1]{IY12} to show that, when $d=2$, the model percolates for large values of $\mu $ as soon as $\lambda>\tilde{\lambda}_c(2\sm)$.

\begin{lemma}\label{le:d2}
If $d=2$ and $g$ has bounded support, then $\mu_c(\lambda,g)<\infty$ for $\lambda>\tilde{\lambda}_c(2\sm)$.
\end{lemma}

\begin{proof}
In the proof of \cite[Theorem 2.1]{IY12}, the result is proved for the Poisson Boolean version of the model, where two vertices within distance $2\sm$ of each other are for sure connected if there is a group in the intersection of the balls with radius $\sm$ centered at the respective vertices. We need to adapt the argument to take into account that such a connection in our case exists only with a certain probability, which can be made as close to 1 by increasing the number of groups in the intersection. To this end, consider the two dimensional lattice with vertex set $m\mathbb{Z}^2$. Let $e=(z,z')$ be an edge of the lattice and consider the rectangle $R_e$ formed by the two squares $S_z$ and $S_{z'}$ with side length $m$ centered at $z$ and $z'$.

Fix $\lambda > \tilde\lambda_c(2\sm)$. By \cite[Theorem 3.7]{MR}, the critical value $\tilde\lambda_c(r)$ is continuous in $r$. It follows that there exists $s_1 <\sm$ such that $\lambda > \tilde\lambda_c(2 s_1)$. Let $s_2 \in (s_1, s_{max})$. Also write $\tilde{\cG}_\cV(\lambda,r)$ for the graph generated by the standard Poisson Boolean model on $\cV$. Next we define some events associated to the edge $e$:
\begin{equation}\nonumber
\begin{split}
    A_e = & \{ \text{there exist a crossing of $R_e$ along its longest side and crossings of $S_z$ and $S_{z'}$ along} \\
    &\text{ the perpendicular sides by a component of $\tilde{\cG}(\lambda, 2s_1)$}  \}
\end{split}
\end{equation}
\begin{equation}\nonumber
\begin{split}
    B_e = &\{\text{for any two vertices $x,y\in \mathcal{V} \cap R_e$ such that $B_x(s_1) \cap B_y(s_1) \not = \emptyset$ there exists} \\ &\text{ a group in $\mathcal{U} \cap B_x(s_2) \cap B_y(s_2)$ to which both to $x$ and $y$ are connected in $\cGbi$}\}.
\end{split}
\end{equation}

\begin{figure}
\begin{subfigure}{0.45\textwidth}
         \centering
         \includegraphics[width=\textwidth]{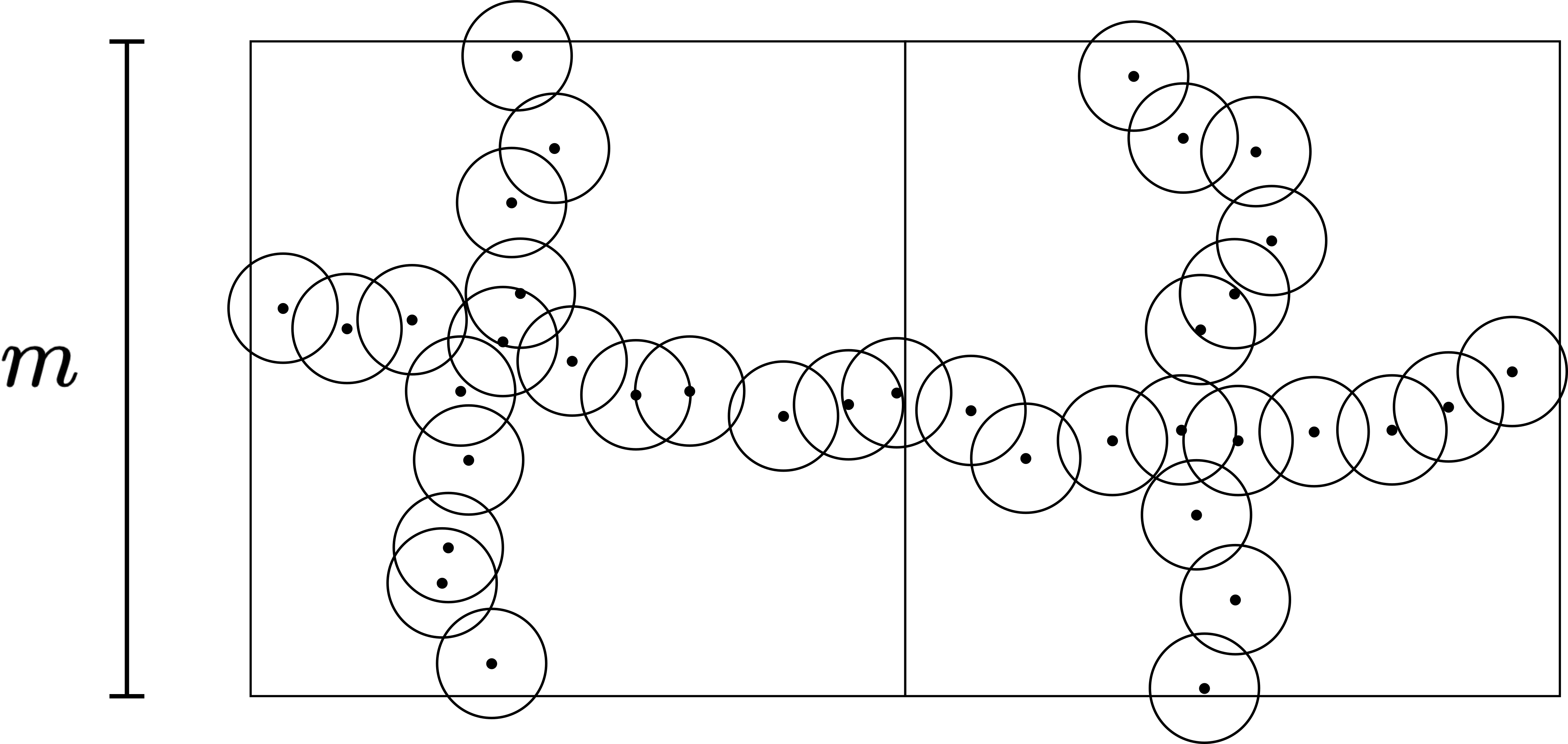}
         \caption{The event $A_e$.}
         \label{fig:crossings}
     \end{subfigure}
     \hfill
     \begin{subfigure}{0.35\textwidth}
         \centering
         \includegraphics[width=\textwidth]{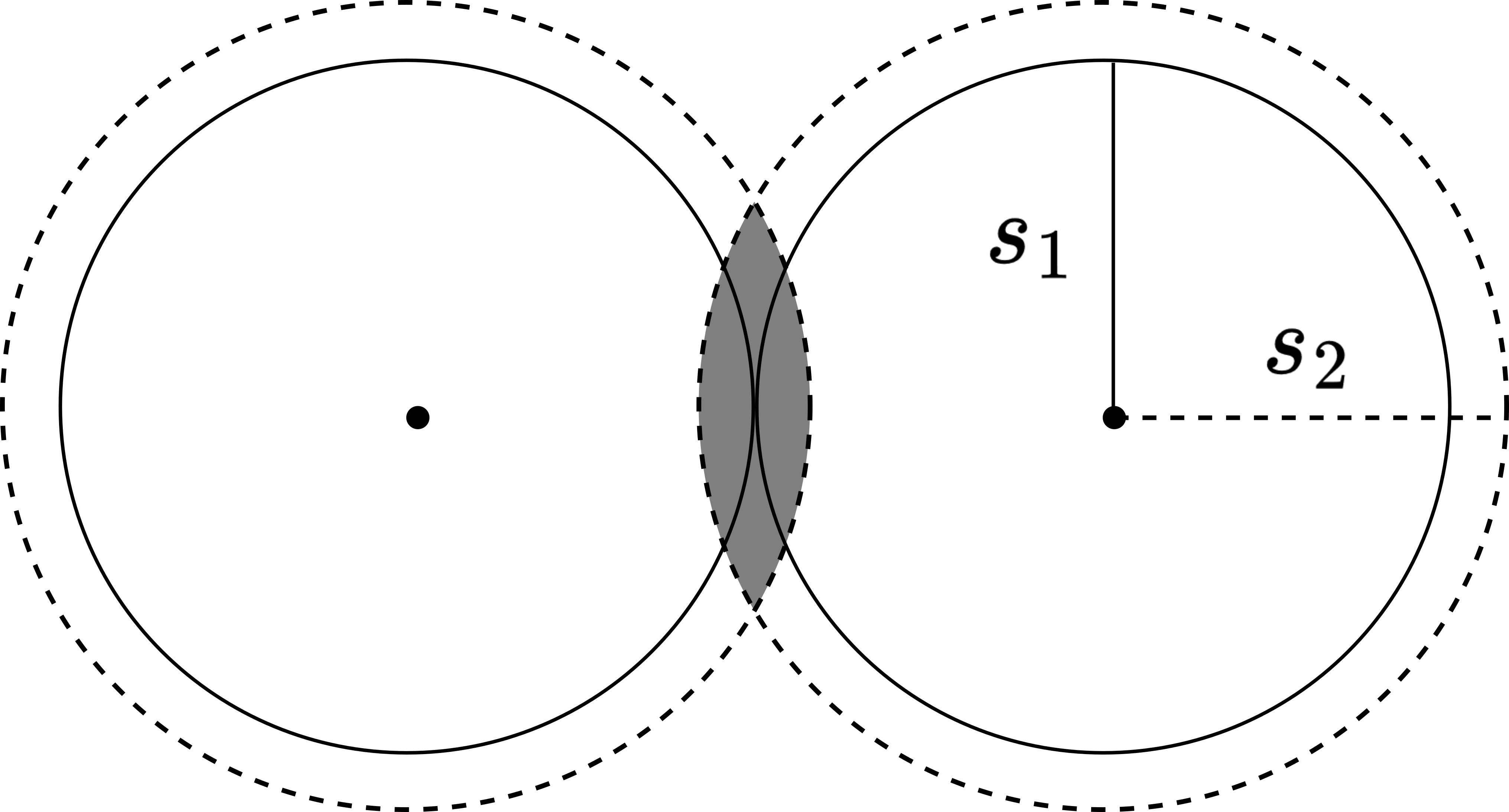}
         \caption{The area $b(s_1,s_2)$}
         \label{fig:balls}
     \end{subfigure}
\caption{Geometrical constructions in the proof of Lemma \ref{le:d2}.}
        \label{fig:geometry}
\end{figure}

See Figure \ref{fig:crossings} for an illustration of the event $A_e$. The edge $e$ is declared open if both $A_e$ and $B_e$ occur. The rest of the proof is divided in two steps: First we prove that the model just defined percolates provided that $\mu$ and $m$ are large enough, then we show that this implies percolation in $\cG_\cV$.

\textbf{Step 1.} Observe that the edges in the lattice form a 2-dependent process. Therefore, from Lemma \ref{le:LSS} we know that percolation in the lattice is achieved whenever $\mathbb{P}(A_e \cap B_e) > p_c$. To prove the latter, we introduce an auxiliary event associated to the edge $e$:
$$
V_e = \{\text{the number of vertices $|\mathcal{V} \cap R_e|$ is smaller than $k$}\}.
$$
Trivially
$$
\PP(A_e \cap B_e) > \PP(A_e \cap B_e \cap V_e) =  \PP(B_e | A_e \cap V_e) \cdot \PP(A_e \cap V_e).
$$

Observe that $\PP(A_e \cap V_e) > \PP(A_e) + \PP(V_e) - 1$. From \cite[Corollary 4.1]{MR} we know that for any $\varepsilon > 0$ there exists $m = m(\varepsilon)$ such that $\PP(A_e) > 1-\varepsilon$, since $\lambda>\tilde{\lambda}_c(2s_1)$ and the graph $\tilde{\cG}(\lambda, 2s_1)$ percolates. Moreover, the number of vertices in the rectangle $|\mathcal{V} \cap R_e|$ is Poisson distributed with mean $\lambda \cdot \ell(R_e) = 2\lambda m^2$. Thus, there exists $k = k(\varepsilon)$ such that $\PP(V_e) > 1-\varepsilon$. Therefore, $\PP(A_e \cap V_e) > 1 - 2\varepsilon$ for any choice of $\varepsilon$, provided that $m$ and $k$ are large enough.

Next, we bound $\mathbb{P}(B_e | A_e \cap V_e)$. Given the event $A_e \cap V_e$, we know that inside $R_e$ there are at most $k^2/2$ pairwise intersections of the balls $\{B_x(s_1)\}_{x \in \mathcal{V} \cap R_e}$. Furthermore, if the balls $B_x(s_1)$ and $B_y(s_1)$ intersect, the area $|B_x(s_2) \cap B_y(s_2)|$ can be bounded from below by $b(s_1,s_2)$, where $b(s_1,s_2)$ is the area of the lens of intersection of two balls with radius $s_2$ whose centers are at distance $2 s_1$; see Figure \ref{fig:balls}. Also observe that, for a group located in $B_x(s_2) \cap B_y(s_2)$, the probability that there is a connection to both $x$ and $y$ (in $\cGbi)$ is lower bounded by $g(s_2)^2$, since the distance to both $x$ and $y$ is at most $s_2$. The number of groups in a given intersection $B_x(s_2) \cap B_y(s_2)$, with $x,y\in\mathcal{V} \cap R_e$, is hence stochastically larger than a Poisson variable with parameter $\mu g(s_2)^2b(s_1,s_2)$, implying that
$$
\PP(B_e | A_e \cap V_e) \geq \left(1 - \exp\left(-\mu g(s_2)^2 b(s_1,s_2)\right)\right)^{k^2/2}.
$$
It follows that $\PP(B_e | A_e \cap V_e)$ can be made as close to 1 as we wish by increasing $\mu$. Summing up, for any $\varepsilon > 0$, there exist $m, \mu, k$ such that $\PP(A_e \cap B_e \cap V_e) > 1 - \varepsilon$. \\

\textbf{Step 2.} Fix $m$ and $\mu$ such that the edge percolation model on $m\Z^2$ defined above percolates. Suppose $e_1, e_2$ are two adjacent edges in the infinite component, that is, suppose the events $A_{e_1}$, $A_{e_2}$, $B_{e_1}$ and $B_{e_2}$ occur. Write $L_{e_1}$ and $L_{e_2}$ for the crossings of the respective rectangles along the longest side. Then, as consequence of the existence of crossings (induced by balls with radius $s_1$) stipulated by $A_{e_1}$ and $A_{e_2}$, we have that either the two crossings $L_{e_1},L_{e_2}$ intersect (if the edges $e_1$ and $e_2$ have different orientations) or they both intersect a perpendicular crossing in the square $R_{e_1} \cap R_{e_2}$ (if $e_1$ and $e_2$ have the same orientation). Now draw balls of radius $s_2$ around each vertex of the crossings $L_{e_1}$, $L_{e_2}$ (and eventually the perpendicular crossing). The events $B_{e_1}$ and $B_{e_2}$ imply that every pairwise intersection of these balls contains a group that is connected to the vertices at the center of the two balls in the graph $\cGbi$. It follows that all vertices in the crossings $L_{e_1}$ and $L_{e_2}$, as well as the perpendicular crossing, belong to the same component in $\cG_\cV$. Therefore, $\cG_\cV$ percolates whenever the edge percolation model on $m \mathbb{Z}^2$ does.
\end{proof}

\begin{remark}
Lemma \ref{le:d2} can be applied also when $g$ has unbounded support to conclude that $\mu_c(\lambda,g)<\infty$ for any $\lambda>0$ in $d=2$ (that is, the conclusion of Lemma \ref{le:perc_ubdd} in $d=2$). Indeed, for a fixed $\lambda>0$, we can take $r<\infty$ such that $\lambda>\tilde{\lambda}_c(2r)$ and Lemma \ref{le:d2} then shows that a model with the connection function $g$ truncated at $r$ gives rise to percolation, which implies percolation in the non-truncated model as well.
\end{remark}

Theorem \ref{th:main} now follows by combining the above lemmas.

\begin{proof}[Proof of Theorem \ref{th:main}]
Part (a) follows from Lemma \ref{le:no_perc} and Lemma \ref{le:perc_ubdd}. As for part (b), the fact that percolation is not possible when $\lambda<\tilde{\lambda}_c(2\sm)$ follows by recalling from Lemma \ref{le:equivalence} and its proof that, if there is an infinite component in $\cG_\cV$, then there is an infinite component in $\cGbi$ where vertices and groups alternate. This means that there is an infinite component in the Poisson Boolean model with vertex set $\cV$ and $r=2\sm$, since two consecutive vertices (with one intermediate group) on the infinite path in $\cGbi$ must be within distance $2\sm$ from each other. But such a component does not exist when $\lambda<\tilde{\lambda}_c(2\sm)$. The fact that $\mu_c(\lambda,g)<\infty$ for any $d\geq 2$ when $\lambda$ is large follows from Lemma \ref{le:perc_bdd}, and Lemma \ref{le:d2} asserts that $\lambda>\tilde{\lambda}_c(2\sm)$ suffices in $d=2$.
\end{proof}

\section{Examples and simulations}

In this section we give some examples of connection functions $g(x)$ and illustrate the behavior of the model by aid of simulations. Recall that the expected number of groups that a vertex is a member of is $\mu ||g||$ and that the expected group size is $\lambda ||g||$. The magnitude of $||g||$ hence plays an important role in determining the sparsity/density of the graph. To provide a fair qualitative comparison of different profile functions, we will therefore sometimes normalize $g$ so that $||g||=1$. Recall the definition \eqref{f} of the self-convolution $f$ and the role of $f$ in the connection probability and the expected degree described in Corollary \ref{cor:conn_prob} and Proposition \ref{p:edegree}. Note that, when $||g||=1$, the function $g$ can be interpreted as the probability density function (pdf) of a random vector $Y$ on $\mathbb{R}^d$ and $f$ is then the pdf of the sum of two independent realizations of $Y$.

We mention three examples of profile functions:
\begin{enumerate}
    \item \textit{Poisson Boolean model.} We first recall the Poisson Boolean model, where
    $$
     g(x) = \mathbbm{1}_{\{|x| < r\}}, \quad r>0.
    $$
    The corresponding random intersection graph $\mathcal{G}_\mathcal{V}$ is the AB Poisson Boolean model studied in \cite{IY12}. Write $B_r(v)$ for the ball with radius $r$ centered at $v$. For any $v \in \mathcal{V}$, we have that
    $$
    g(x)g(v-x) = \begin{cases}
    1 &\text{if $x \in B_r(0) \cap B_r(v)$} \\
    0 &\text{otherwise}.\end{cases}
    $$
    Hence $f(v) = |B_r(0) \cap B_r(v)|$ and the probability that $0$ and $v$ are connected is $p_{0,v} = 1 - \exp\{-\mu |B_r(0) \cap B_r(v)|\}$, which equals 0 when $|v|>2r$.
    \item \textit{Normal distribution.} An example of connection probabilities with unbounded support is provided by
    $$
    g(x) = \frac{1}{(2\pi \sigma^2)^{d/2}} e^{-\frac{|x|^2}{2 \sigma^2}}.
    $$
     Here $g$ is the pdf of a normal random vector with $||g||=1$ and $f$ is hence the pdf of the sum of two independent normal random vectors with pdf $g$, that is,
    $$
    f(v) = \frac{1}{(4\pi \sigma^2)^{d/2}} e^{-\frac{|v|^2}{4 \sigma^2}}.
    $$
    The probability that two vertices $0$ and $v$ are connected is $p_{0,v} = 1 - e^{-\mu f(v)}$, and we note that $p_{0,v} \to 1$ as $\mu \to \infty$ and decays exponentially to 0 when $|v| \to \infty$.


    \item \textit{Power-law decay.} A connection probability function with polynomially decaying tail is given by
    \begin{equation}
    g(x) = 1 \wedge |x|^{-d \alpha}, \quad \alpha > 1.
    \end{equation}
    In the (discrete version of the) standard random connection model, this is known as long-range percolation; see references in Section \ref{sec:further}. The function $f$ is given by
    $$
    f(v) = \int_{\mathbb{R}^d} (1 \wedge |x|^{-d \alpha})(1 \wedge |v-x|^{-d \alpha}) dx.
    $$
    We do not have an explicit expression for $f$ in this case, but it can easily be estimated numerically.
\end{enumerate}

\textbf{Visualization.} Figure \ref{fig:visual} contains visualizations of the random intersection graph $\cG_\cV$ for the above choices of connection probabilities, scaled so that the average degree in the graphs is similar. The model has been simulated on the torus $\mathbb{T}$, represented as a square where opposite sides are identified. Vertices and groups are positioned on the torus according to two independent Poisson processes with intensity $\lambda=2$ and $\mu=1$, respectively. For the Poisson Boolean function and the normal distribution most edges are relatively short, due to the fact that the connection probability $g$ decays exponentially fast with the distance. For the power-law distribution with $\alpha = 3/2$, we observe more long-range connections in the graph.

\textbf{Degree distribution.} Figures \ref{fig:degree_g} and \ref{fig:degree_D0} show the degree distribution for a few different instances of the normal distribution. Again, the model is simulated on a torus. Figure \ref{fig:degree_g} illustrates the importance of $||g||$: In Figure \ref{fig:degree_g}a, we have $||g||=1$, while $||g||=4$ in Figure \ref{fig:degree_g}b, resulting in a degree distribution shifted towards larger degrees. In both cases, the degree distribution appears Poisson-like, apart from the spike at degree 0, most evident in Figure \ref{fig:degree_g}a. This spike is explained in that a vertex that does not belong to any groups is for sure isolated, giving rise to the bound $\PP(D=0)\geq 1 - e^{-\mu ||g||}$, which is larger for smaller values of $||g||$ and $\mu$. Figure \ref{fig:degree_D0} shows the degree distribution for $||g||=1$ with the roles of $\lambda$ and $\mu$ reversed in the two pictures. Indeed, the proportion of isolated vertices is larger for the smaller value of $\mu$ in Figure \ref{fig:degree_D0}a.

\textbf{Percolation properties.} Next we illustrate the critical value $\mu_c$ as a function of $\lambda$ and $g$. We simulate $\cG_\cV$ on a torus and compute the size of the largest component of $\cG_\cV$, expressed as a proportion of the total number of vertices. Large values of this proportion indicate percolation, whereas small values indicate no percolation. When we vary $\lambda$ and $\mu$, the critical value $\mu_c$ is approximated by the curve separating the two regions in the $\lambda-\mu$ plane. In Figure \ref{fig:perc_standard}, the function $g$ is the standard normal distribution on $\mathbb{R}^2$ and $\lambda$ and $\mu$ vary between $0$ and $4$. We note that the two percolation phases are clearly visible and observe that $\mu_c(\lambda)$ looks symmetric with respect to the bisector of the plane, which is consistent with the exchangeability of $\lambda$ and $\mu$ from Lemma \ref{le:equivalence}.

According to Theorem \ref{th:main}, the critical value $\mu_c$ is always finite if $g$ has unbounded support, whereas $\mu_c = \infty$ for small values of $\lambda$ if $g$ has bounded support. This is reflected in Figure \ref{fig:perc_support}. Indeed, when $g$ is the standard normal distribution on $\mathbb{R}^2$, we see that percolation seems achievable also for small values of $\lambda$ by increasing $\mu$ sufficiently, while in the Poisson Boolean, there is a value of $\lambda$ below which percolation seems unfeasible.

\begin{figure}[p]
     \centering
     \begin{subfigure}[b]{0.48\textwidth}
         \centering
         \includegraphics[width=\textwidth]{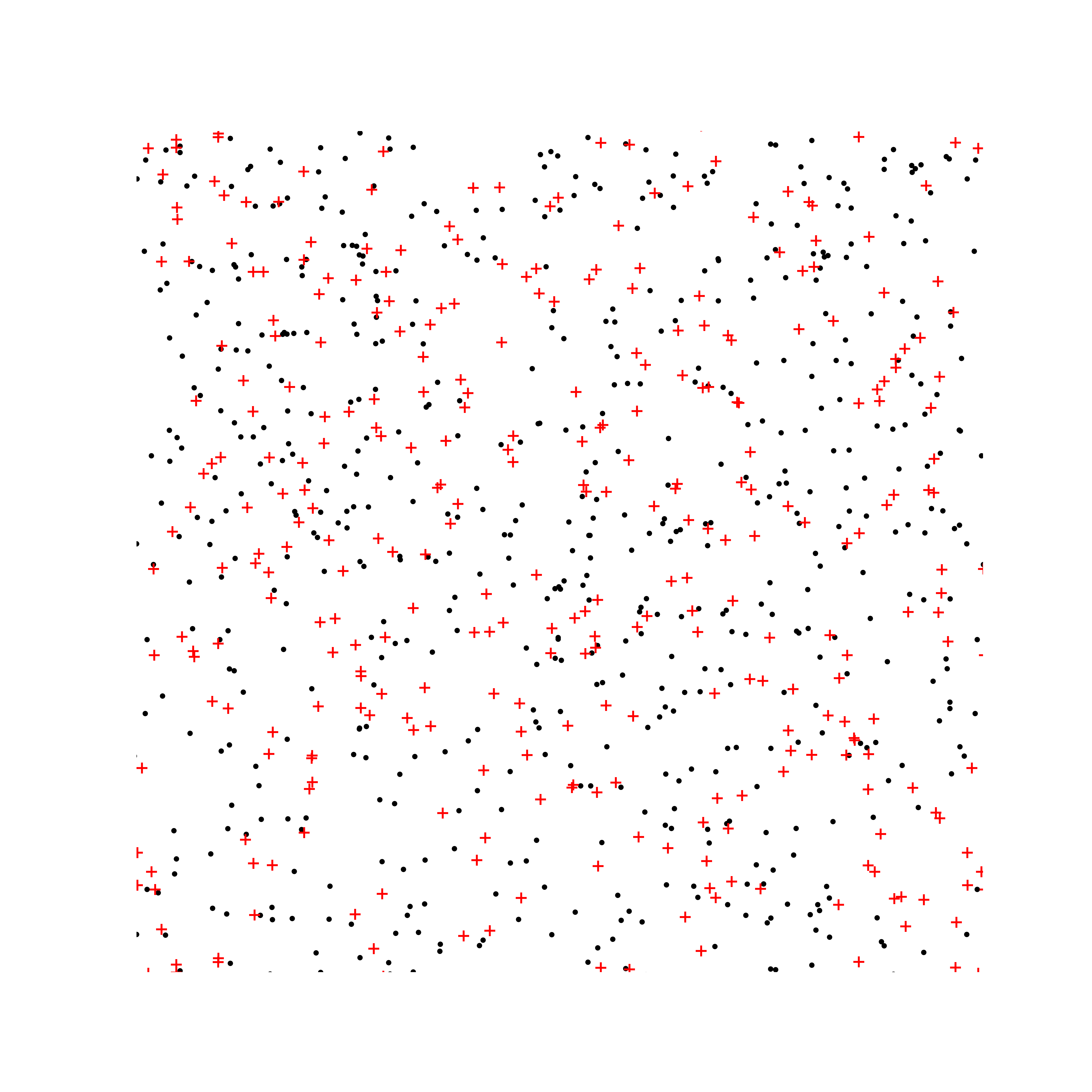}
         \caption{$\mathcal{V}$ and $\mathcal{U}$}
         \label{fig:a}
     \end{subfigure}
     \hfill
     \begin{subfigure}[b]{0.48\textwidth}
         \centering
         \includegraphics[width=\textwidth]{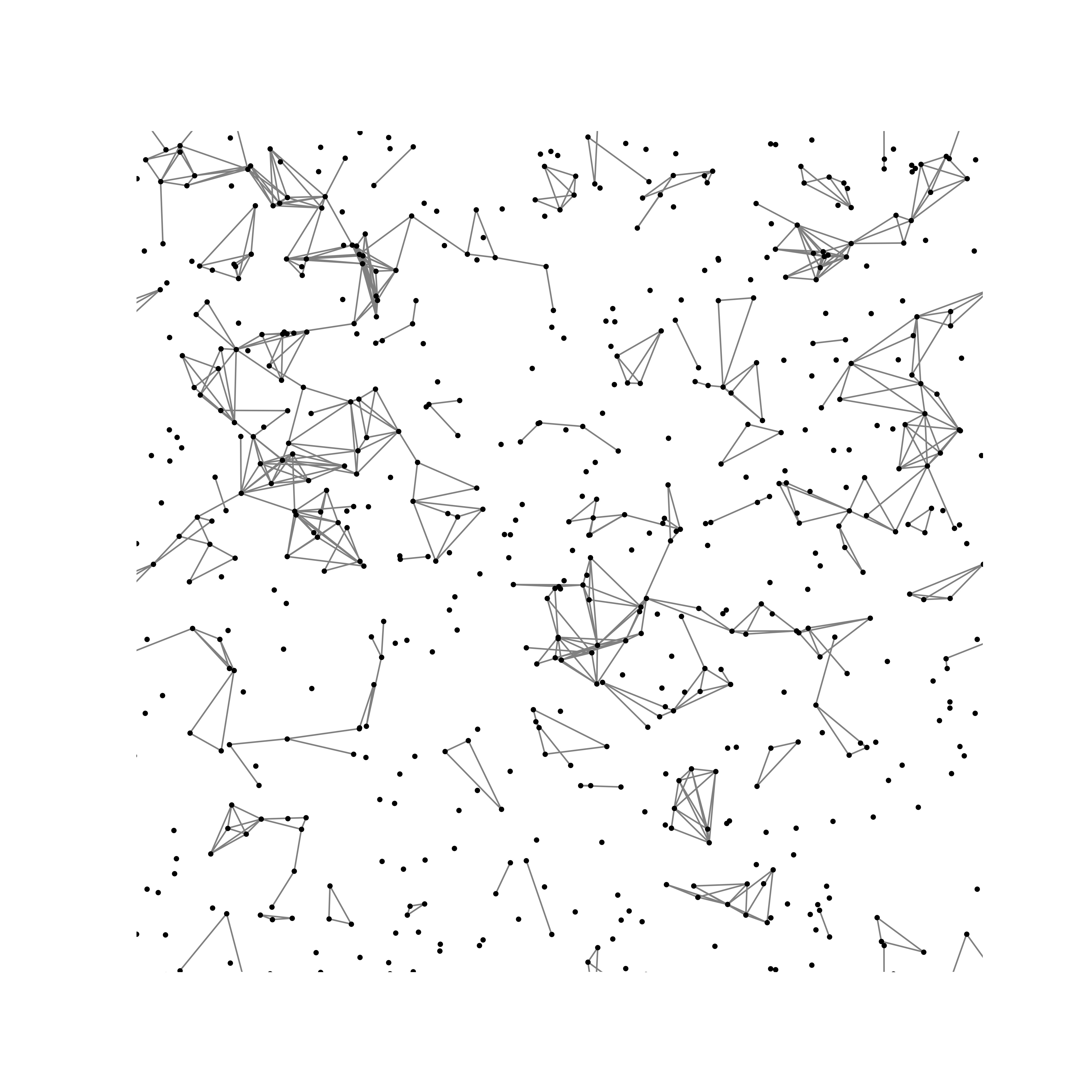}
         \caption{$g(x) \propto \mathbbm{1}_{\{|x| < 1\}}$}
         \label{fig:b}
     \end{subfigure}
     \\
     \vspace{.5cm}
     \begin{subfigure}[b]{0.48\textwidth}
         \centering
         \includegraphics[width=\textwidth]{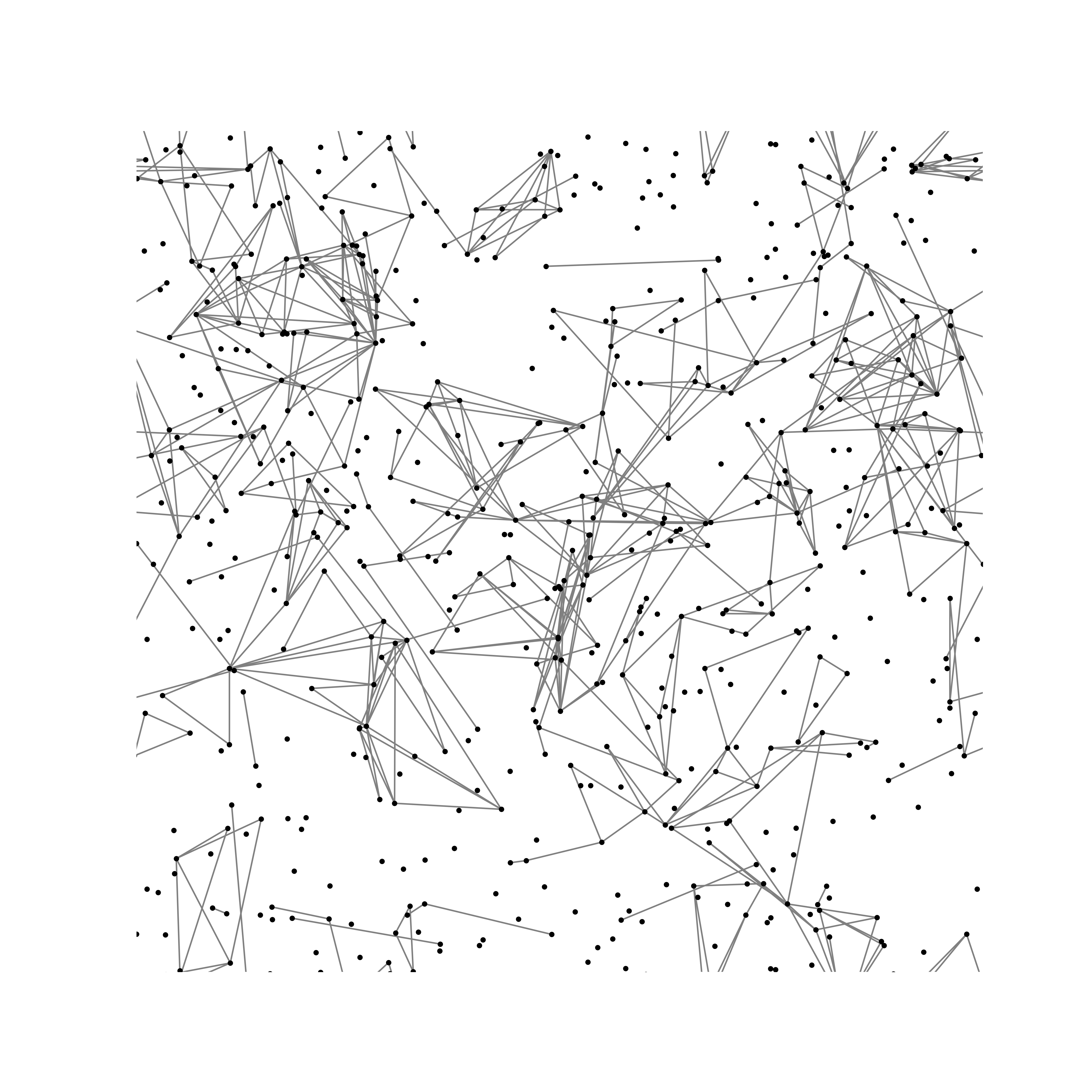}
         \caption{$g(x) \propto e^{-\frac{|x|^2}{2}}$}
         \label{fig:c}
     \end{subfigure}
     \hfill
     \begin{subfigure}[b]{0.48\textwidth}
         \centering
         \includegraphics[width=\textwidth]{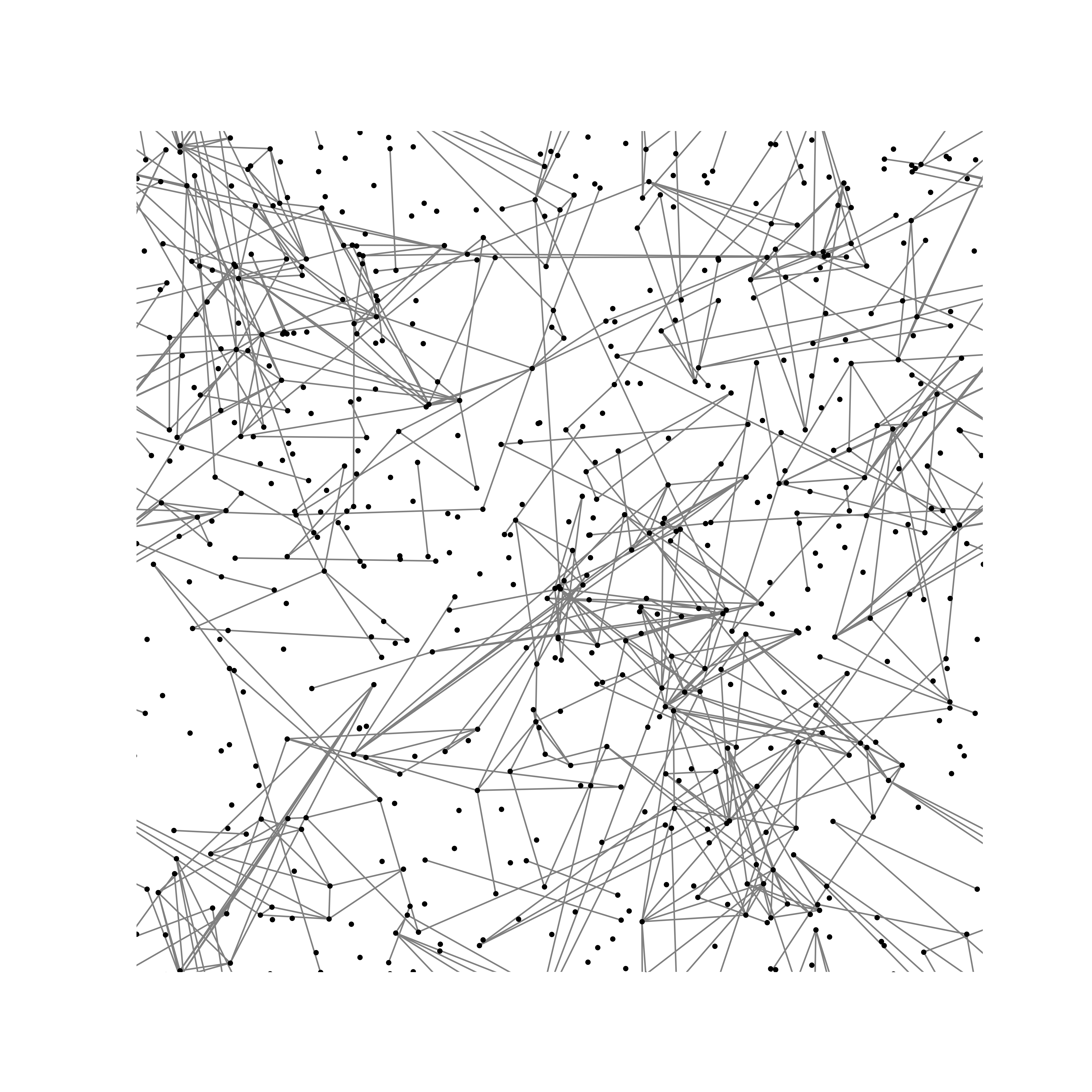}
         \caption{$g(x) \propto 1 \wedge |x|^{-3}$}
         \label{fig:d}
     \end{subfigure}
        \caption{Visualization of $\mathcal{G}_{\mathcal{V}}$ for different choices of connection probabilities. The vertices and groups are sampled with intensities $\lambda = 2, \mu = 1$ on a torus of size $3 \times 10^2$. Fig. \ref{fig:a} shows the positions of vertices (black dots) and groups (red crosses). Fig. \ref{fig:b}, \ref{fig:c}, \ref{fig:d} then shows the graph $\cG_\cV$ for the indicated connection probabilities.}
        \label{fig:visual}
\end{figure}

\begin{figure}[p]
     \begin{subfigure}[b]{0.45\textwidth}
         \centering
         \includegraphics[width=\textwidth]{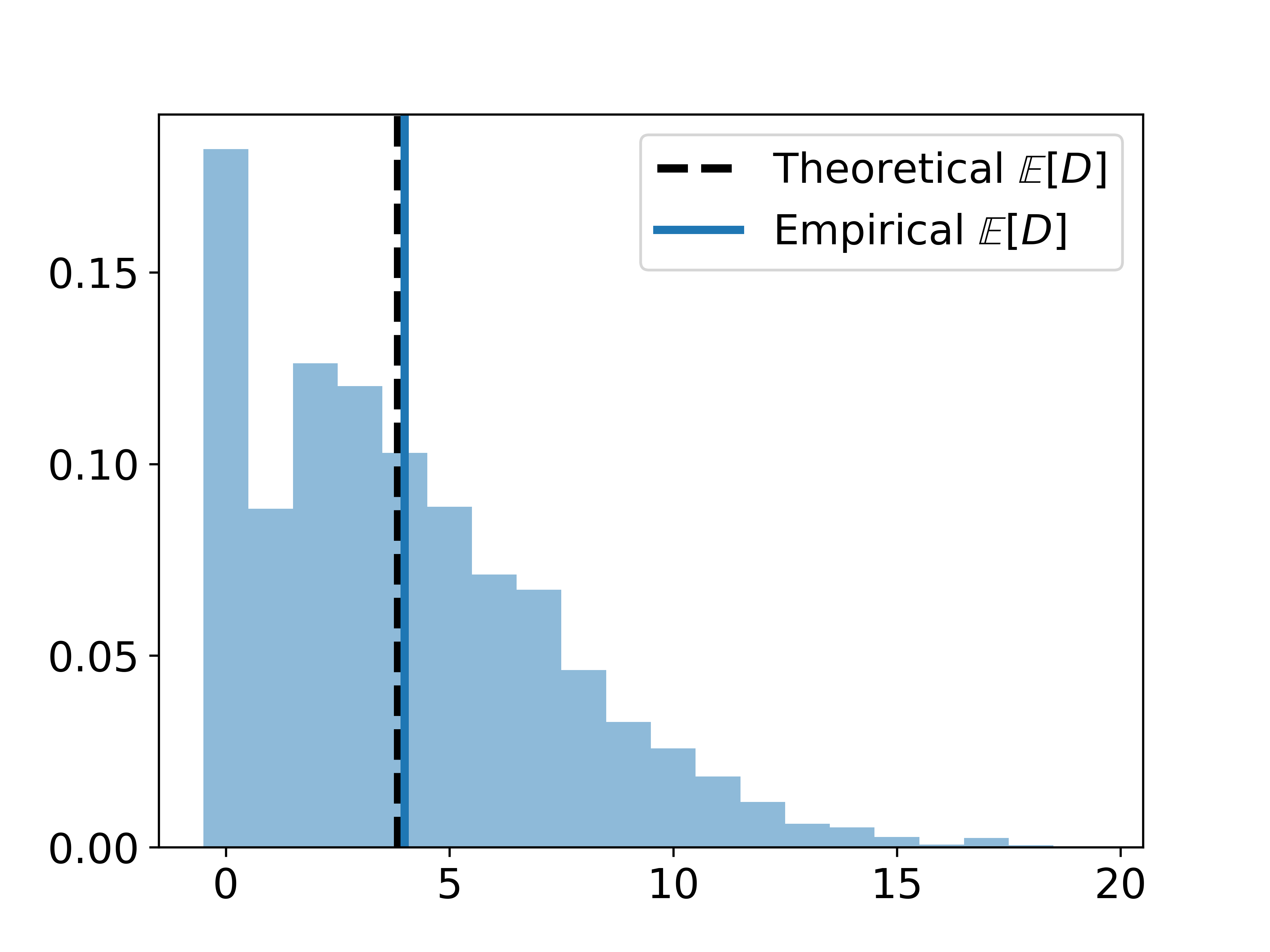}
         \caption{$g(x) \propto e^{-|x|^2/2}, ||g|| = 1$ \\ $\lambda=2, \mu=2$}
         \label{fig:degree_g1}
     \end{subfigure}
     \hfill
     \begin{subfigure}[b]{0.45\textwidth}
         \centering
         \includegraphics[width=\textwidth]{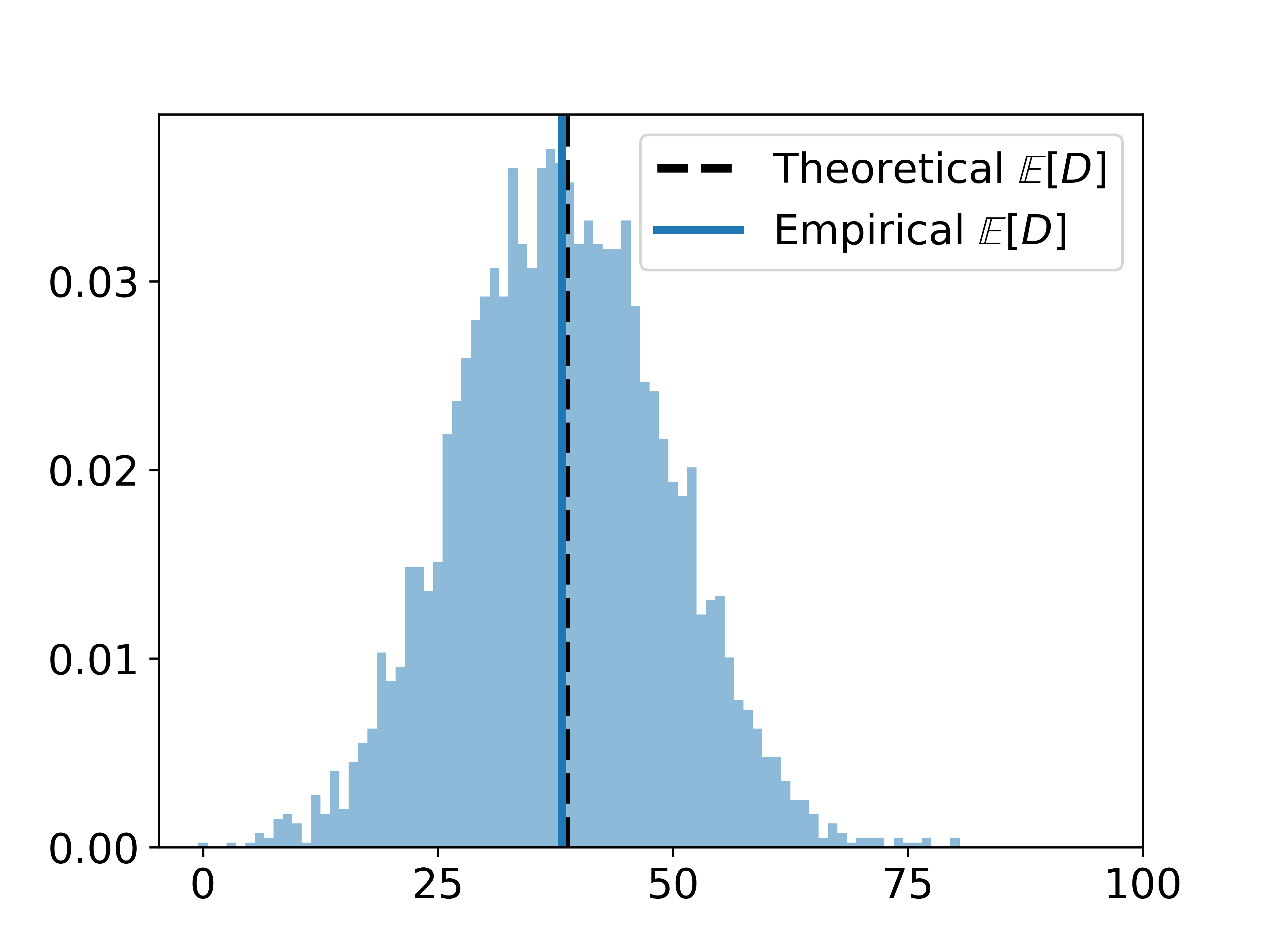}
         \caption{$g(x) \propto e^{-|x|^2/2}, ||g|| = 4$ \\ $\lambda=2, \mu=2$}
         \label{fig:degree_g4}
     \end{subfigure}
     \caption{Degree distribution of $\cG_\cV$ sampled on a torus of size $2 \times 10^3$. The solid line is the value of the empirical expected degree, whereas the dotted line is the value of the theoretical expected degree numerically computed from the expression in Proposition \ref{p:edegree}. The intensities $\lambda$ and $\mu$ are the same in Fig. \ref{fig:degree_g1} and \ref{fig:degree_g4}, but the value of $||g||$ is different.}
        \label{fig:degree_g}
\end{figure}

\begin{figure}[p]
     \begin{subfigure}[b]{0.45\textwidth}
         \centering
         \includegraphics[width=\textwidth]{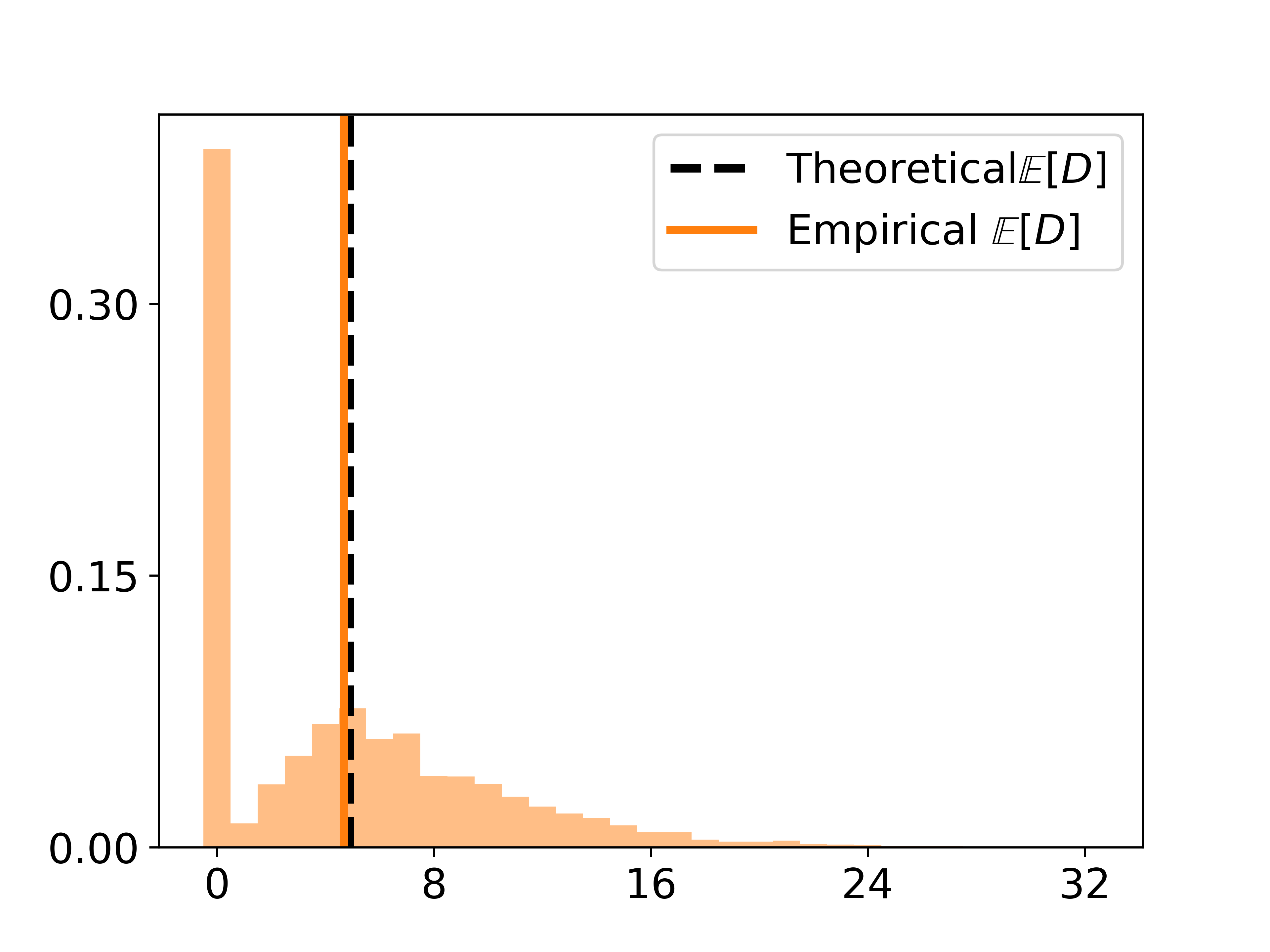}
         \caption{$g(x) \propto e^{-|x|^2/2}, ||g|| = 1$ \\ $\lambda=5, \mu=1$}
         \label{fig:degree_l5m1}
     \end{subfigure}
     \hfill
     \begin{subfigure}[b]{0.45\textwidth}
         \centering
         \includegraphics[width=\textwidth]{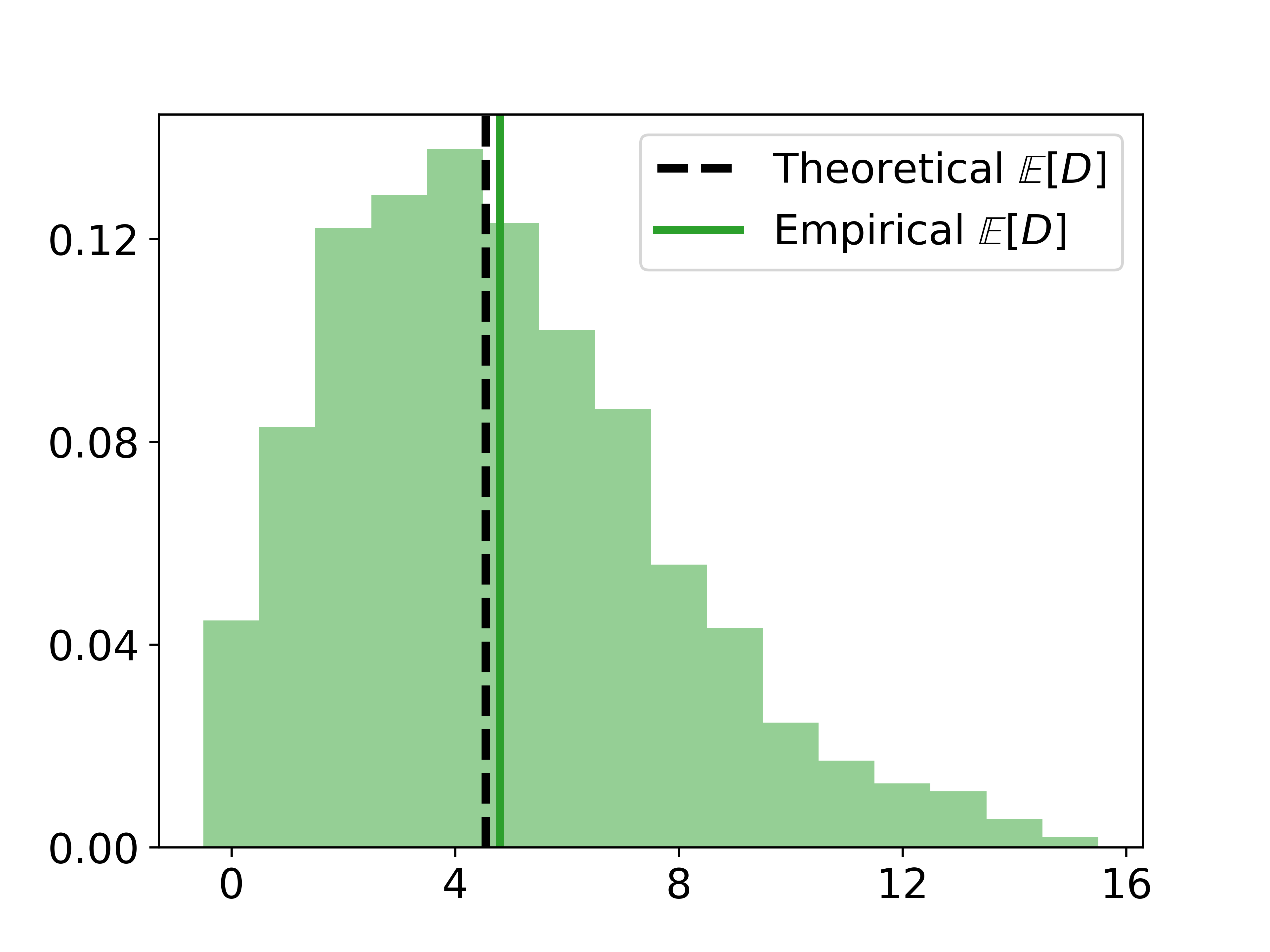}
         \caption{$g(x) \propto e^{-|x|^2/2}, ||g|| = 1$ \\ $\lambda=1, \mu=5$}
         \label{fig:degree_l1m5}
     \end{subfigure}
     \caption{Degree distribution of $\cG_\cV$ sampled on a torus of size $2 \times 10^3$. The solid line is the value of the empirical expected degree, whereas the dotted line is the value of the theoretical expected degree numerically computed from the expression in Proposition \ref{p:edegree}. The intensities $\lambda$ and $\mu$ are different in Fig. \ref{fig:degree_l5m1} and \ref{fig:degree_l1m5}.}
        \label{fig:degree_D0}
\end{figure}

\begin{figure}[p]
    \centering
     \begin{subfigure}[b]{0.5\textwidth}
         \centering
         \includegraphics[width=\textwidth]{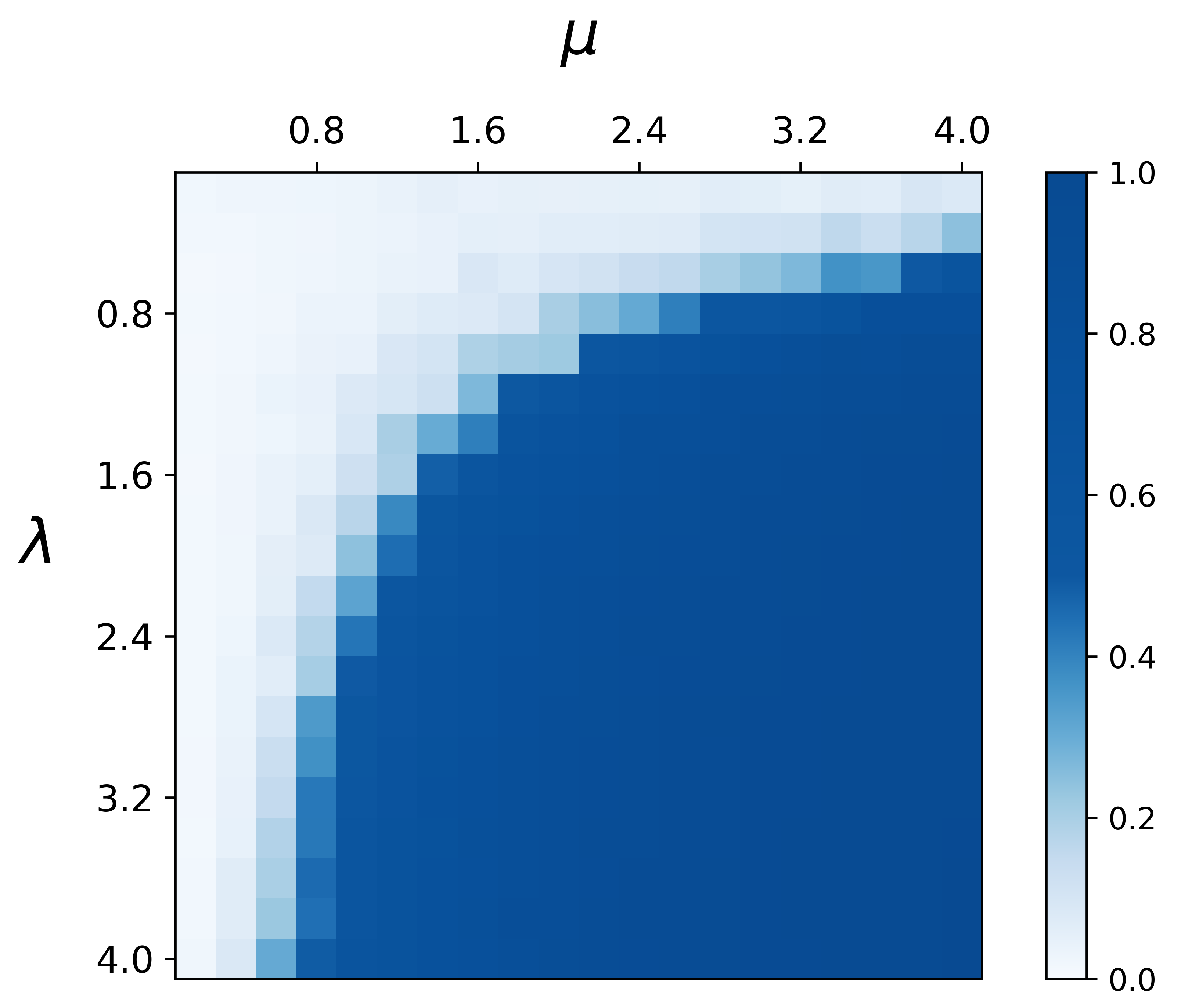}
     \end{subfigure}
     \caption{Percolation phases for $\cG_\cV$, simulated on the torus of size $10^3$ using $g(x) \propto \frac{1}{2\pi}e^{-|x|^2/2}$. The value in each square is the proportion of vertices in the largest component of the graph, obtained as an average over 10 samples: in the blue region the proportion of vertices in the largest component is large (percolation); in the white region the proportion is small (no percolation).}
        \label{fig:perc_standard}
\end{figure}

\begin{figure}[p]
    \centering
    \begin{subfigure}[b]{0.6\textwidth}
         \centering
         \includegraphics[width=\textwidth]{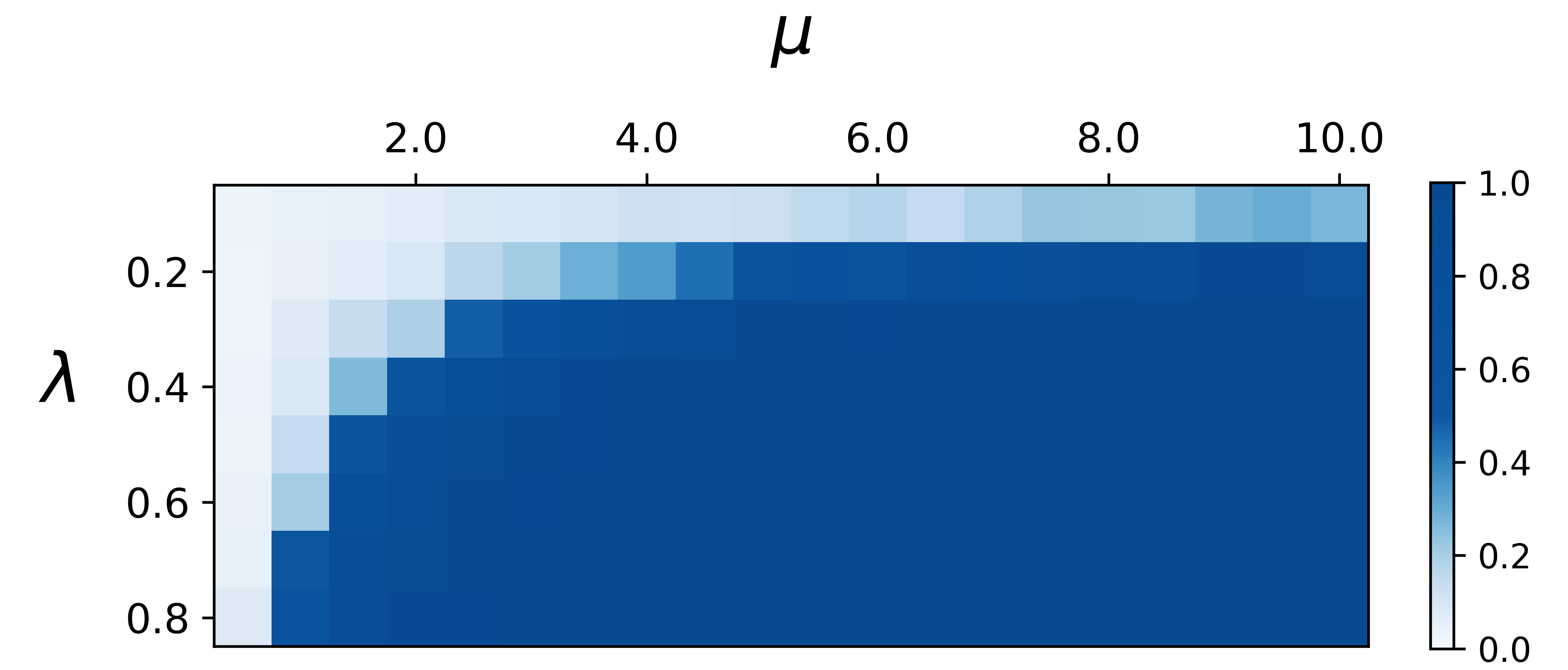}
         \caption{$g(x) \propto e^{-|x|^2/2}$}
         \label{fig:perc_support_unbounded}
     \end{subfigure}
     \\
     \vspace{0.5cm}
     \begin{subfigure}[b]{0.6\textwidth}
         \centering
         \includegraphics[width=\textwidth]{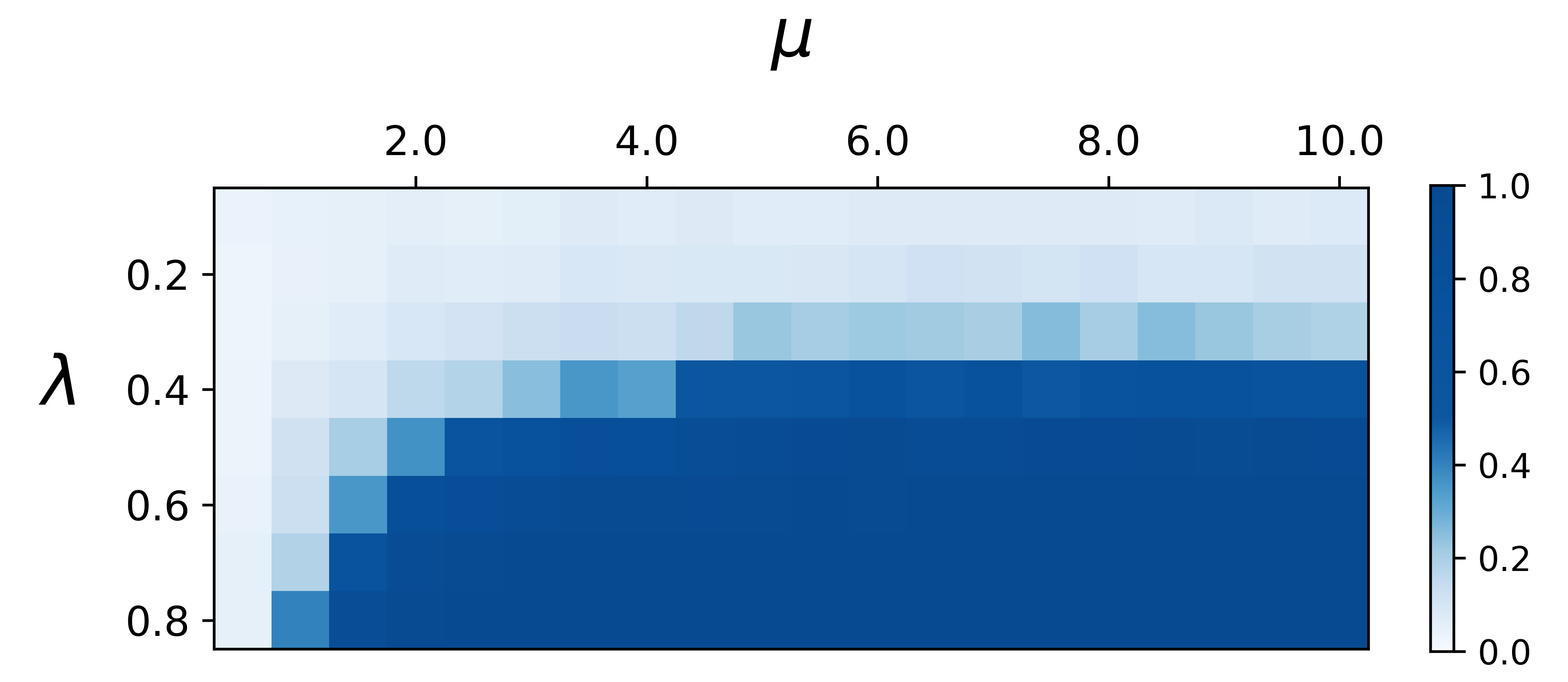}
         \caption{$g(x) \propto \mathbbm{1}_{\{|x|<1\}}$}
         \label{fig:perc_support_bounded}
     \end{subfigure}
     \caption{Percolation phases for $\cG_\cV$, simulated on the torus of size $10^3$ averaging the proportion of vertices in the largest component over 10 samples of the graph. In Fig. \ref{fig:perc_support_unbounded} $g$ has unbounded support; in Fig. \ref{fig:perc_support_bounded} the support of $g$ is bounded.}
        \label{fig:perc_support}
\end{figure}

\pagebreak

\textbf{Funding information}

This work was carried out during a research visit of R.\ Michielan at Stockholm University spring 2023, funded by the Erasmus+ KA131 European Mobility project. The work of M.\ Deijfen is supported in part by the Swedish Research Council.

\end{document}